\documentclass[12pt]{amsart}
\usepackage{amssymb,latexsym,eufrak,amsmath,amscd, graphicx}
\usepackage[colorlinks=true,pagebackref,hyperindex]{hyperref}

\usepackage{amsfonts}
\usepackage{amscd}

\renewcommand{\phi}{\varphi}
\renewcommand{\epsilon}{\varepsilon}
\renewcommand{\theta}{\vartheta}

\def\ZZ{{\mathbf Z}}
\def\NN{{\mathbf N}}
\def\FF{{\mathbf F}}
\def\CC{{\mathbf C}}
\def\AAA{{\mathbf A}}
\def\RR{{\mathbf R}}
\def\QQ{{\mathbf Q}}
\def\PP{{\mathbf P}}

\def\cJ{\mathcal{J}}

\def\cO{\mathcal{O}}

\def\fra{\mathfrak{a}}
\def\frb{\mathfrak{b}}

 \DeclareMathOperator{\Spec}{Spec}

\newcommand{\llbracket}{[\negthinspace[}
\newcommand{\rrbracket}{]\negthinspace]}

\newtheorem{lemma}{Lemma}[section]
\newtheorem{theorem}[lemma]{Theorem}
\newtheorem{corollary}[lemma]{Corollary}
\newtheorem{proposition}[lemma]{Proposition}

\newtheorem{conjecture}[lemma]{Conjecture}

\theoremstyle{definition}

\newtheorem{remark}[lemma]{Remark}

\newtheorem{notation}[lemma]{Notation}
\newtheorem{example}[lemma]{Example}
\newtheorem{question}[lemma]{Question}

\theoremstyle{remark}
\newtheorem*{remark*}{Remark}
\newtheorem*{note*}{Note}

\begin{document}

\title{Bernstein-Sato polynomials in positive characteristic}

\thanks{2000\,\emph{Mathematics Subject Classification}.
 Primary 13A35; Secondary 14B05, 32S40.
\newline The author
  was partially supported by NSF grant DMS DMS-0758454  and
  by a Packard Fellowship}

\author[M. Musta\c{t}\v{a}]{Mircea~Musta\c{t}\v{a}}
\address{Department of Mathematics, University of Michigan,
Ann Arbor, MI 48109, USA} \email{{\tt mmustata@umich.edu}}

\begin{abstract}
In characteristic zero, the Bernstein-Sato polynomial of a
hypersurface can be described as the minimal polynomial of the
action of an Euler operator on a suitable $D$-module. We consider
 analogous $D$-modules in positive characteristic, and use them to
define a sequence of Bernstein-Sato polynomials (corresponding to
the fact that we need to consider also divided powers Euler
operators). We show that the information contained in these
polynomials is equivalent to that given by the $F$-jumping exponents
of the hypersurface, in the sense of Hara and Yoshida \cite{HY}.
\end{abstract}

\maketitle

\markboth{M.~Musta\c t\u a}{Bernstein-Sato polynomials in positive
characteristic}

\section{Introduction}

The goal of this note is to describe a connection between the theory
of generalized test ideals, in the sense of Hara and Yoshida
\cite{HY}, and the theory of $D$-modules. Suppose that $X={\rm
Spec}(R)$ is a nonsingular affine scheme, and that $f\in R$ is a
nonzero regular function on $X$.

Let us describe first the characteristic zero situation, studied by
Malgrange in \cite{Malgrange}. If $\iota\colon X\to X\times \AAA^1$
is the graph of $f$, let $B_f:=\iota_+\cO_X$ denote the $D$-module
theoretic push-forward of the structure sheaf of $X$. This has the
following explicit description as the first local cohomology module
of $X\times\AAA^1$ along the image of $\iota$
\begin{equation}\label{eq0_1}
B_f\simeq H^1_{\iota(X)}\cO_{X\times\AAA^1}\simeq R[t]_{f-t}/R[t].
\end{equation}
 The class of $1/(f-t)$ in $B_f$ is denoted by $\delta$ (this is the
$\delta$-function corresponding to the graph of $f$). Let $D_R$
denote the ring of differential operators on $R$. Malgrange
constructed the $V$-filtration on $B_f$, which is a filtration by
$D_R$-modules such that, informally speaking, $\partial_tt$ is put
in upper-triangular form when passing to the graded module
associated to this filtration. The key ingredient in this
construction is the Bernstein-Sato polynomial $b_f(s)\in\QQ[s]$,
that in this context can be interpreted as the minimal polynomial of
$-\partial_tt$ acting on the $D_R$-module
\begin{equation}\label{eq0_2}
M_f/tM_f,\,\text{where}\,\,M_f:=D_R[\partial_t t]\cdot\delta.
\end{equation}

We mention that a recent result of Budur and Saito \cite{BS} relates
the $V$-filtration to the theory of multiplier ideals as follows.
Recall that for every nonnegative $\lambda$ one defines the
multiplier ideal $\cJ(f^{\lambda})\subseteq \cO_X$, and one gets in
this way a decreasing filtration of $\cO_X$ (see Chap. 9 in
\cite{Laz}). If one considers the embedding $R\hookrightarrow B_f$
given by $h\to h\delta$, then the $V$-filtration induces (up to a
minor renormalization) the filtration on $R$ given by the multiplier
ideals of $f$.

Suppose now that ${\rm char}(R)=p>0$, and let us assume that $R$ is
$F$-finite, that is, the Frobenius morphism on $R$ is finite. In
this case, the ring of differential operators $D_R$ is not finitely
generated over $R$, but it can be written as a union of subrings
$D_R^e$, where $D_R^e={\rm End}_{R^{p^e}}(R)$.

Our main point is that one can define the $D_R$-modules $B_f$ and
$M_f$ also in the positive characteristic setting, and these $D_R$-modules are
related to the generalized test ideals $\tau(f^{\lambda})$ of Hara
and Yoshida \cite{HY}. As in the case of multiplier ideals,
$\lambda$ is a nonnegative real parameter. The generalized test ideals give a
decreasing filtration of $R$, and the exponents where the test
ideals change value are the \emph{F-jumping exponents} of $f$. It
was shown in \cite{BMSm1} that the $F$-jumping exponents of $f$ form
a discrete set of rational numbers. We stress that unlike the
multiplier ideals that are defined via a resolution of
singularities, the test ideals are defined using the action of the
Frobenius morphism on the ring. On the other hand, there are
interesting results and conjectures relating the multiplier ideals
and the test ideals via reduction mod $p$.

Note that in characteristic $p>0$ we have an infinite set of Euler
operators $\theta_{p^i}:=\partial_t^{[p^i]}t^{p^i}$, for $i\geq 0$
(recall that  $\partial_t^{[m]}$ is the differential operator whose
action is given by $\partial_t^{[m]}\bullet t^r= {{r}\choose {m}}
t^{r-m}$). Unlike in characteristic zero, the action of these
operators on $D$-modules is easy to describe.
In fact,  every $D_{R[t]}^e$-module admits 
a decomposition into common eigenspaces for the operators
$\theta_1,\theta_p,\ldots,\theta_{p^{e-1}}$  (the eigenvalues being in $\FF_p$). 
In the case of the module $B_f$, we write down an
$R$-basis of $B_f$ consisting of common eigenvectors. Moreover, the action of
$D_R$, $t$ and $\partial_t^{[p^i]}$ on this basis can be described
explicitly (see Theorem~\ref{thm4_1} for the precise statement).

Instead of only considering the $D_R$-module $M_f/t M_f$, in this
case it is natural to consider separately all modules
\begin{equation}\label{eq0_3}
M_f^e/t M_f^e,
\,\text{where}\,\,M_f^e=D_R^e[\theta_1,\ldots,\theta_{p^{e-1}}]\cdot\delta.
\end{equation}
The corresponding eigenspace decomposition for $B_f$ induces a
decomposition of $M_f^e/tM_f^e$ into common eigenspaces for
$\theta_1,\theta_p,\ldots,\theta_{p^{e-1}}$, each eigenvalue lying
in $\FF_p$.

By analogy with the characteristic zero situation, we define the
Bernstein-Sato polynomial of $f$ to be the minimal polynomial of
$-\theta_1$ acting on $M_f^1/tM_f^1$. This is a product of linear
forms in $\FF_p[s]$, each appearing with multiplicity one. Note that
if $f$ is the reduction mod $p\gg 0$ of a polynomial
$\widetilde{f}\in\ZZ[x_1,\ldots,x_n]$, then $b_f$ divides the reduction mod $p$
of $b_{\widetilde{f}}$ (since $b_{\widetilde{f}}\in\QQ[x]$, the reduction mod $p$ of $b_{\widetilde{f}}$ makes
sense if $p$ is large enough).

In order to also keep track of the higher Euler operators it is more
convenient to consider Bernstein-Sato polynomials with coefficients
in $\QQ$. We put $b_f^{(1)}(s):=\prod_i \left(s-\frac{i}{p}\right)$,
where the product is over those $i\in \{0,1,\ldots,p-1\}$ such that
there is a nonzero eigenvector of $-\theta_1$ in $M_f^1/t M_f^1$
with eigenvalue $\overline{i}\in\FF_p$. More generally, for every
$e\geq 1$ we put
$$b_f^{(e)}(s):=\prod_{i_1,\ldots,i_{e}}\left(s-\left(\frac{i_e}{p}+\ldots+\frac{i_1}{p^e}\right)\right),$$
where the product is over those
$i_1,\ldots,i_e\in\{0,1,\ldots,p-1\}$ such that there is a nonzero
$w\in M_f^e/t M_f^e$ with $(\theta_{p^{\ell-1}} +i_{\ell})w=0$ for
every $1\leq\ell\leq e$. In other words, the Bernstein-Sato
polynomial $b_f^{(e)}$ describes the common eigenvalues of the
operators $\theta_1,\theta_p,\ldots,\theta_{p^{e-1}}$ acting on
$M_f^{e}/t M_f^{e}$.

Our main result says that the information given by the polynomials
$b_f^{(e)}$ is equivalent to that of the $F$-jumping exponents of
$f$. If $\lambda>1$, then $\lambda$ is an $F$-jumping exponent if
and only if $\lambda-1$ has this property, and therefore it is
enough to understand the $F$-jumping exponents in the interval
$(0,1]$ (recall that this is a finite set of rational numbers). In
the next theorem, we denote by $\lceil u\rceil$ the smallest integer
$\geq u$.

\noindent{\bf Theorem}. Let $R$ be a regular $F$-finite ring of
positive characteristic $p$. Consider the $F$-jumping exponents
$\lambda_1,\ldots,\lambda_r$ of $f$ that lie in $(0,1]$. 
Given $e\geq 1$, the rational number $\frac{\lceil p^e\lambda_i\rceil-1}{p^e}$
is a root of the Bernstein-Sato polynomial $b_f^{(e)}$. Moreover, every root of $b_f^{(e)}$ is of this form,
for some $i\in\{1,\ldots,r\}$.

We mention that the first connection between invariants in positive
characteristic and Bernstein-Sato polynomials has been noticed in
\cite{MTW}. With the above notation, the result in \emph{loc. cit.}
can be stated as follows. Suppose that $\widetilde{f}$ is defined over $\ZZ$,
and that $f$ is the reduction mod $p$ of $\widetilde{f}$, for some $p\gg 0$.
If $b_{\widetilde{f}}$ is the Bernstein-Sato polynomial of $\widetilde{f}$, and if $\lambda$
is an $F$-jumping exponent of $f$, then ${\lceil
p^e\lambda\rceil-1}$ is a root of $b_{\widetilde{f}}$ mod $p$. The above theorem
is a first step towards a better understanding of this connection
between Bernstein-Sato polynomials and the generalized test ideals.

\bigskip

The paper is structured as follows. The first two sections are of an
expository nature, reviewing the necessary notions from zero and
positive characteristic. In \S 2 we give an introduction to the
circle of ideas around the $V$-filtration. In particular, we
describe the role of the Bernstein-Sato polynomial in this setting.
In \S 3 we overview the definition of the generalized test ideals
following \cite{BMSm2}. We also discuss the most interesting
results and conjectures about these ideals, concerning their
connection with multiplier ideals via reduction mod $p$. In \S 4 we
show that every $D_{R[t]}^e$-module has a canonical decomposition into
common eigenspaces with respect to the action of the Euler operators
$\theta_1,\theta_p,\ldots,\theta_{p^{e-1}}$.
In \S 5 we turn to the case of the module $B_f$, and we write down 
an explicit basis of common eigenvectors.
 In the last section we
define the Bernstein-Sato polynomials and prove the above theorem.
We end with some questions related to this setup.

\subsection*{Acknowledgement}
I am grateful to Manuel Blickle, Nero Budur, and Morihiko Saito
for many discussions and comments related to this project. I am also indebted to 
Claude Sabbah for suggesting that one should consider the decomposition with respect
to Euler operators for arbitrary $D$-modules in positive characteristic.

\section{Bernstein-Sato polynomials and $V$-filtrations}

We recall in this section, following \cite{Malgrange}, the notion of
$V$-filtration and its connection with the Bernstein-Sato
polynomial. We work over a fixed algebraically closed field $k$ of
characteristic zero. For simplicity, we restrict to the hypersurface
case, though a similar picture is known to hold for ideals of
arbitrary codimension (see \cite{BMSa}).

Let $X$ be a smooth, connected $n$-dimensional variety, and let $H$
be a hypersurface in $X$. Our invariants are local, hence we may and
will assume that $X=\Spec(R)$ is affine and $H$ is defined by
$(f=0)$ for some nonzero $f\in R$. We denote by $D_R$ the ring of
differential operators on $X$ (over $k$), and denote by $P\bullet h$
the action of $P\in D_R$ on $h\in R$. Around every point in $X$ we
can find a principal affine open subset $U=\Spec(R_a)$ such that we
have $x_1,\ldots,x_n\in R_a$ that give an \'{e}tale morphism $U\to
\AAA^n$. If $\partial_1,\ldots,\partial_n\in {\rm Der}_k(R_a)$ are
the corresponding derivations, then $D_{R_a}\simeq (D_R)_a$ is
generated by $R_a$ and $\partial_1,\ldots,\partial_n$.

We now give the definition of the Bernstein-Sato polynomial.
Consider an extra variable $s$, and the free $R_f[s]$-module
generated by the symbol $f^s$. This is, in fact, a left module over
$D_{R_f}[s]$ if we let a derivation $D$ of $R_f$ act by
$$D\cdot f^s=\frac{s D(f)}{f}f^s.$$
It was shown by Bernstein that there is a nonzero $b(s)\in k[s]$ and
$P\in D_R[s]$ (that is, $P$ is a polynomial in $s$ with coefficients
in $D_R$) such that
\begin{equation}\label{bernstein}
b(s)f^s=P\cdot f^{s+1}.
\end{equation}
It is clear that the set of polynomials $b(s)$ for which there is
$P$ satisfying (\ref{bernstein}) is an ideal in $k[s]$. The monic
generator of this ideal is called the \emph{Bernstein-Sato
polynomial} of $f$, and it is denoted by $b_f$.

In (\ref{bernstein}) we have treated $f^s$ as a formal symbol.
However, this equation has the obvious meaning whenever we can make
sense of $f^s$. For example, if $m\in\ZZ$, we can let $s=m$ in
(\ref{bernstein}) and then we get a corresponding equality in $R_f$.

The Bernstein-Sato polynomial is a subtle invariant of the
singularities of the hypersurface $H=(f=0)$. A deep theorem of
Kashiwara \cite{Kashiwara} says that all roots of $b_f$ are negative
rational numbers. In particular, $b_f$ has rational coefficients.
One of the main applications of the $V$-filtration in
\cite{Malgrange} was to relate, when $H$ has isolated singularities,
the roots of $b_f$ with the eigenvalues of the monodromy action on
the Milnor fiber.

We now explain the definition of the $V$-filtration of $f$, and the
connection with the Bernstein-Sato polynomial. Let $\iota\colon
X\hookrightarrow X\times\AAA^1$ be the graph map of $f$, that is
$\iota(x)=(x,f(x))$. We have a left $D$-module on $X\times \AAA^1$
(that is, a left $D_{R[t]}$-module) given as the $D$-module
push-forward of $R$, namely $B_f:=\iota_+R$. This can be explicitly
described as the first cohomology module of $X\times\AAA^1$ along
the graph of $f$
$$B_f\simeq R[t]_{f-t}/R[t].$$
Via this identification, the action of the differential operators on
$B_f$ is induced by the natural action on the localization of
$R[t]$. It is easy to see that if we denote by $\delta$ the class of
$\frac{1}{f-t}$ in $B_f$, then $B_f$ is free over $R$ with a basis
given by
$$\partial_t^m\cdot \delta=\frac{m!}{(f-t)^{m+1}},$$
for $m\geq 0$.

Consider now the $D_R$-module $M_f:=D_R[\partial_t
t]\cdot\delta\subset B_f$. One can show that
$$t M_f=D_R[\partial_t t]\cdot t\delta=
D_R[\partial_t t]\cdot f\delta\subseteq M_f.$$ A key observation is that
(\ref{bernstein}) holds if and only if
\begin{equation}\label{bernstein2}
b(-\partial_t t)\cdot\delta=P(-\partial_t t)f\cdot \delta.
\end{equation}
Indeed, consider the ring homomorphism $\phi\colon D_{R_f}[s]\to
D_{R_f}[\partial_t t]$ given by $\phi(P(s))=P(-\partial_t t)$. This
makes $B_f\otimes_RR_f$ into a $D_{R_f}[s]$-module. We also have a
$D_{R_f}[s]$-linear map $\psi\colon R_f[s]f^s\to B_f\otimes_RR_f$
given by $\psi(Q(s)f^s)=Q(-\partial_t t)\cdot \delta$. To see that $\psi$
is indeed linear with respect to the action of differential
operators, note that if $D$ is a derivation on $R_f$, then
$$\psi(D\cdot f^s)=\psi\left(\frac{sD(f)}{f}f^s\right)=-\partial_t
t\cdot \frac{D(f)}{f}\delta=-D(f)\partial_t\cdot \delta=D\cdot
\delta.$$ Since $\{(\partial_t t)^m\cdot \delta\}_{m\geq 0}$ gives a
basis of $B_f\otimes_RR_f$ over $R_f$, it follows that $\psi$ is
injective. Using also the fact that $B_f\subseteq B_f\otimes_RR_f$,
we deduce that (\ref{bernstein}) is equivalent with
(\ref{bernstein2}).

Moreover, it is easy to see that $b(-\partial_t t)\cdot\delta \in t
M_f$ if and only if $b(-\partial_t t)\cdot M_f\subseteq t M_f$. We
conclude that $b_f$ is the minimal polynomial of the action of
$-\partial_t t$ on $M_f/tM_f$.

\medskip

The $V$-filtration is a decreasing filtration on $B_f$ by finitely generated left
$D_R$-submodules $\{V^{\alpha}\}_{\alpha\in\QQ}$, with the following
properties:
\begin{enumerate}
\item[(i)] $\bigcup_{\alpha\in\QQ}V^{\alpha}=B_f$.

\item[(ii)] The filtration is semicontinuous and discrete in the following
sense: there is a positive integer $\ell$ such that for every
integer $m$ and every $\alpha\in
\left(\frac{m-1}{\ell},\frac{m}{\ell}\right]$ we have
$V^{\alpha}=V^{m/\ell}$.

\item[(iii)]  We have $t\cdot V^{\alpha}\subseteq V^{\alpha+1}$ for every
$\alpha$, with equality if $\alpha>0$.

\item[(iv)] We have $\partial_t\cdot V^{\alpha}\subseteq V^{\alpha-1}$
for every $\alpha$.

\item[(v)] For every $\alpha$, if we put
$V^{>\alpha}:=\bigcup_{\beta>\alpha}V^{\beta}$, then
$(\partial_tt-\alpha)$ is nilpotent on $V^{\alpha}/V^{>\alpha}$.
\end{enumerate}

The key property is (v) above. One can consider the $V$-filtration
as an attempt to put the operator $\partial_tt$ on $B_f$ in upper
triangular form. It is not hard to show that if a filtration as
above exists, then it is unique. Malgrange proved the existence of
such a filtration in \cite{Malgrange}, using only the existence of
the Bernstein-Sato polynomial and the rationality of its roots.

\medskip

There is, in fact, an explicit description of the $V$-filtration in
terms of more general Bernstein-Sato polynomials, due to Sabbah
\cite{Sabbah}. One can show (for example, using the existence of the
$V$-filtration) that for every $w\in B_f$ there is a nonzero
polynomial $b(s)\in\QQ[s]$ and $P\in D_R[s]$ such that
\begin{equation}\label{bernstein3}
b(-\partial_tt)w=P(-\partial_tt)t\cdot w.
\end{equation}
The set of polynomials $b(s)$ for which there is $P$ as above is an
ideal, and its monic generator is called the Bernstein-Sato
polynomial of $f$ associated to $w$, and it is denoted by $b_{f,w}$. Note that we
have $b_f=b_{f,\delta}$. It is a consequence of the existence of the
$V$-filtration that all roots of $b_{f,w}$ are rational. Using this
terminology, Sabbah showed that $V^{\alpha}$ is the subset of $B_f$
consisting of those $w$ such that all roots of $b_{f,w}$ are $\leq
-\alpha$.

\bigskip

We end this section with a result of Budur and Saito relating the
$V$-filtration to the multiplier ideals of $f$. Recall that given
$f$, we can use a log resolution of singularities for the pair
$(X,H)$ to attach to every $\lambda\in\RR_+$ an ideal in $R$ called
the \emph{multiplier ideal} of $f$ of exponent $\lambda$, and
denoted by $\cJ(f^{\lambda})$. We refer to \cite{Laz}, Chap.~9 for
the precise definition and for the basic properties. If
$\lambda>\mu$, then $\cJ(f^{\lambda})\subseteq\cJ(f^{\mu})$.
Moreover, given $\lambda\in\RR_+$, there is $\epsilon>0$ such that
$\cJ(f^{\lambda})=\cJ(f^{\lambda+\epsilon})$. A \emph{jumping
exponent} of $f$ is a positive $\lambda$ such that
$\cJ(f^{\lambda})\varsubsetneq\cJ(f^{\lambda-\epsilon})$ for every
$\epsilon>0$. We make the convention that $0$ is also a jumping
exponent. It follows easily from definition that all jumping
exponents are rational and that they form a discrete subset of
$\RR_+$. Since we consider only principal ideals, we also have
$\cJ(f^{\lambda+1})=f\cdot\cJ(f^{\lambda})$, hence $\lambda$ is a
jumping exponent if and only if $\lambda+1$ is.

Note that we have an embedding $R\hookrightarrow B_f$ given by $h\to
h\delta$. Budur and Saito showed in \cite{BS} that the restriction
to $R$ of the $V$-filtration is, essentially, the filtration of $R$
by multiplier ideals. More precisely, they showed that for every
$\lambda\in\RR_+$ we have $\cJ(f^{\lambda})=V^{\lambda+\epsilon}\cap
R$ for $0<\epsilon\ll 1$. One deduces as an easy consequence of
their statement the following result from \cite{ELSV}: if
$\lambda\in (0,1]$ is a jumping exponent of $f$, then
$b_f(-\lambda)=0$. Note also that in light of Sabbah's description
of the $V$-filtration, the result of Budur and Saito can be
reinterpreted as saying that for $h\in R$ we have
$$\sup\{\alpha\in\RR_+\mid h\in \cJ(f^{\alpha})\}=-\max\{\beta\mid
b_{f,h\delta}(\beta)=0\}.$$

\section{Generalized test ideals}

Hara and Yoshida introduced in \cite{HY} a characteristic $p$
analogue of the multiplier ideals, the (generalized) test ideals. In
this section we recall the definition of these ideals, and their
connection with the multiplier ideals via reduction mod $p$. In
fact, since our ambient variety is nonsingular, we find it more
convenient to work with an equivalent definition from \cite{BMSm2}.
We stick to the hypersurface case, as in the rest of the paper,
though for most results in this section the extension to the case of
arbitrary ideals is verbatim.

We fix a regular domain $R$ of positive characteristic $p$. We
always assume $R$ to be $F$-finite (that is, the Frobenius
homomorphism $F\colon R\to R$ given by $F(u)=u^p$ is finite). Note
that since $R$ is regular, $F$ is also flat, hence $R$ is locally
free over $R^p$. Basic examples are $k[x_1,\ldots,x_n]$ or
$k\llbracket x_1,\ldots,x_n\rrbracket$, where $k$ is a perfect field
(or more generally, such that $[k\colon k^p]$ is finite).

If $J$ is an ideal in $R$ and $e\geq 1$, we denote by $J^{[p^e]}$
the $e^{\rm th}$ Frobenius power of $J$, that is, the ideal generated by the
$p^e$-powers of the elements in $J$
$$J^{[p^e]}=
(u^{p^e}\mid u\in J).$$ If $\frb$ is an arbitrary ideal, then one can
easily deduce from the fact that $R$ is locally free over $R^{p^e}$
that among the ideals $J$ such that $\frb\subseteq J^{[p^e]}$ there
is a unique minimal one, that we denote by $\frb^{[1/p^e]}$.

If $R$ is free over $R^p$, then one can compute $\frb^{[1/p^e]}$ as
follows. Since $R$ is free also over $R^{p^e}$, we can choose a basis
$y_1,\ldots,y_N$ of $R$ over $R^{p^e}$.
Consider generators $h_1,\ldots,h_r$ of $\frb$, and write for every
$i$
$$h_i=\sum_{j=1}^Na_{i,j}^{p^e}y_j,$$
with $a_{i,j}\in R$. With this notation, we have
$\frb^{[1/p^e]}=(a_{i,j}\mid i,j)$.

We now fix a nonzero $f\in R$ and a nonnegative real number
$\lambda$. One can check using the definition that for every $e\geq
1$ we have
$$\left(f^{\lceil \lambda p^e\rceil}\right)^{[1/p^e]}\subseteq
\left(f^{\lceil \lambda p^{e+1}\rceil}\right)^{[1/p^{e+1}]}.$$ Since
$R$ is Noetherian, it follows that for $e\gg 0$ the ideal
$\left(f^{\lceil\lambda p^e\rceil}\right)^{[1/p^e]}$ does not depend
on $e$. This is the (generalized) \emph{test ideal}
$\tau(f^{\lambda})$. It is easy to see that if $\lambda=m/p^e$ for a
nonnegative integer $m$, then $\tau(f^{\lambda})=(f^m)^{[1/p^e]}$
(see, for example, Lemma~2.1 in \cite{BMSm1}).

Note that $\tau(f^0)=R$. It follows from definition that if
$\lambda>\mu$, then $\tau(f^{\lambda})\subseteq\tau(f^{\mu})$. It is
shown in \cite{BMSm2} that for every nonnegative $\lambda$ there is
$\epsilon>0$ such that
$\tau(f^{\lambda})=\tau(f^{\lambda+\epsilon})$. A positive $\lambda$
is called an \emph{F-jumping exponent} if
$\tau(f^{\lambda})\neq\tau(f^{\lambda-\epsilon})$ for every
$\epsilon>0$. We make the convention that $0$ is also an $F$-jumping
exponent.

It is again easy to see from definition that $\tau(f^{\lambda+1})
=f\cdot \tau(f^{\lambda})$, hence $\lambda$ is an $F$-jumping
exponent if and only if $\lambda+1$ is. Other properties are more
subtle: it is shown in \cite{BMSm1} that every $F$-jumping exponent
is rational, and that the set of $F$-jumping exponents is discrete
in $\RR$ (see also \cite{KLZ}).

\begin{remark}\label{interpretation_thresholds}
We mention an interpretation of the $F$-jumping exponents as
$F$-thresholds (see Proposition~2.7 in \cite{MTW} and Corollary~2.30
in \cite{BMSm2}). Let $J$ be an ideal in $R$ such that $f\in {\rm
Rad}(J)$. For every $e\geq 1$, we denote by $\nu^J(p^e)$ the largest
$r\in\NN$ such that $f^r\not\in J^{[p^e]}$ (if there is no such $r$,
then we put $\nu^J(p^e)=0$).

It is easy to see that we have ${\rm sup}_e\frac{\nu^J(p^e)}{p^e}=
\lim_{e\to\infty}\frac{\nu^J(p^e)}{p^e}<\infty$, and this limit is
called the \emph{F-threshold} of $f$ with respect to $J$, and
denoted by $c^J(f)$. One shows that the set of $F$-jumping exponents
of $f$ is equal to the set $\{c^J(f)\mid f\in {\rm Rad}(J)\}$.

We also note that one can show that if $J\neq R$ (in which case
$c^J(f)>0$), then $\nu^J(p^e)=\lceil c^J(f)p^e\rceil -1$ (see
Proposition~1.9 in \cite{MTW}).
\end{remark}

\bigskip

Arguably the most interesting questions in this area involve the
connections between multiplier ideals and test ideals, via reduction
mod $p$. We now state the fundamental result, due to Hara and Yoshida.

Suppose that $R$ is a domain that is smooth over $\ZZ$ (in
particular, it is of finite type over $\ZZ$) such that
$R\otimes_{\ZZ}\QQ\neq 0$. Let $f\in R$ be nonzero. For every prime
$p$ and every ideal $\fra$ in $R$, we denote by $\fra_p$ the image
of $\fra$ in $R_p:=R\otimes_{\ZZ}\FF_p$, where $\FF_p=\ZZ/p \ZZ$. We
take a log resolution of $(R\otimes_{\ZZ}\QQ,f)$, and we choose
$a\in\ZZ$ such that this resolution is defined over $R\otimes_{\ZZ}\ZZ[\frac{1}{a}]$.
If $p\gg 0$, then we may reduce the resolution mod $p$, such that it
gives a log resolution of $(R_p,f_p)$. In fact, since $p\gg 0$ we
may also assume that the push-forward sheaves that come up in the
construction of multiplier ideals commute with base-change over
$\ZZ$ (note that we essentially deal with finitely many ideals).

\begin{theorem}\label{thm_HY}${\rm (}$\cite{HY}${\rm )}$
With the above notation, we have the following:
\begin{enumerate}
\item[i)] If $p\gg 0$, then for every $\lambda\in\RR_+$ we have
$$\tau(f_p^{\lambda})\subseteq\cJ(f^{\lambda})_p.$$
\item[ii)] Moreover, for every $\lambda$, if $p\gg 0$, then
we have equality
$$\tau(f_p^{\lambda})=\cJ(f^{\lambda})_p.$$
\end{enumerate}
\end{theorem}

The first assertion in the above theorem is proved by interpreting
both the test ideal and the multiplier ideal in terms of local
cohomology. The second part is much more subtle, making use of the
Frobenius action on the de Rham complex, and of vanishing theorems
in positive characteristic, following Deligne and Illusie \cite{DI}.

\begin{conjecture}\label{conj}
With the notation in Theorem~\ref{thm_HY}, there are infinitely many
primes $p$ such that for \emph{all} $\lambda\in\RR_+$ we have
$$\tau(f_p^{\lambda})=\cJ(f^{\lambda})_p.$$
\end{conjecture}

To illustrate the above behavior, we give two examples.

\begin{example}\label{cusp}
We first treat the case of the cusp $f=x^2+y^3\in\ZZ[x,y]$. Because of the
periodicity properties of both multiplier ideals and test ideals, it
is enough to only consider exponents in $[0,1)$. It follows from the
well-known computation of the multiplier ideals of the cusp in
characteristic zero (see Example 9.2.15 in \cite{Laz}) that if $p\gg
0$, then
\[\cJ(f^{\lambda})_p = \left\{
\begin{array}{cl}
\FF_p[x,y], & {\rm for}\,0 \le \lambda < \frac{5}{6}; \\[2mm]
(x,y), & {\rm for}\,\frac{5}{6} \le \lambda < 1.
\end{array}\right.
\]

On the other hand, we claim that if $p>3$, then
\[\tau(f_p^{\lambda}) = \left\{
\begin{array}{cl}
\FF_p[x,y], & {\rm for}\,0 \le \lambda < c(f_p); \\[2mm]
(x,y), & {\rm for}\,c(f_p) \le \lambda < 1,
\end{array}\right.
\]
where $c(f_p)=\frac{5}{6}$ if $p\equiv 1$ (mod 3), and
$c(f_p)=\frac{5}{6}-\frac{1}{6p}$ if $p\equiv 2$ (mod 3).

Indeed, the fact that $\tau(f_p^{\lambda})=\FF_p[x,y]$ if and only
if $\lambda<c(f_p)$ was shown in \cite{MTW}, Example~4.3. In order
to complete the proof of the claim it is enough to show that
$(x,y)\subseteq \left(f^{p^e-1}\right)^{[1/p^e]}$ for every $e\geq
1$. Indeed, since the origin is the only singular point of $f$, it
follows that if $\lambda<1$, then $\tau(f^{\lambda})\subseteq
(x,y)$. On the other hand, if $\lambda<1$ and $e$ is large enough,
then $\lambda p^e\leq p^e-1$, hence
$$(x,y)\subseteq \left(f^{p^e-1}\right)^{[1/p^e]}\subseteq
\left(f^{\lceil \lambda p^e\rceil}\right)^{[1/p^e]},$$ which implies
that $\tau(f^{\lambda})=(x,y)$.

Note first that for every $0\leq a\leq p^e-1$, the binomial
coefficient ${{p^e-1}\choose a}$ is not zero in $\FF_p$. Indeed,
this follows from the fact that the order of $p$ in $m!$ is
$\sum_{i\geq 1}\lfloor m/p^i\rfloor$, and the fact that for every
$1\leq e'\leq e-1$ we have
$$\lfloor a/p^{e'}\rfloor+\lfloor
(p^e-1-a)/p^{e'}\rfloor-\lfloor (p^e-1)/p^{e'}\rfloor
=p^{e-e'}-\lceil (a+1)/p^{e'}\rceil+\lfloor
a/p^{e'}\rfloor-(p^{e-e'}-1)=0.$$

We now compute $\left(f^{p^e-1}\right)^{[1/p^e]}$ by writing
$f^{p^e-1}$ in the basis of $\FF_p[x,y]$ over $\FF_p[x^{p^e}, y^{p^e}]$
given by $\{x^iy^j\vert 0\leq i,j\leq p^e-1\}$. Since the monomial
$$(x^2)^{\frac{p^e-1}{2}}(y^3)^{\frac{p^e-1}{2}}=y^{p^e}\cdot
x^{p^e-1}y^{\frac{p^e-3}{2}}$$ appears with a nonzero coefficient in
$f^{p^e-1}$, we see that $y\in \left(f^{p^e-1}\right)^{[1/p^e]}$.
Since $(x^2)^{p^e-1}=x^{p^e}\cdot x^{p^e-2}$ appears with
coefficient one in $f^{p^e-1}$, we deduce that
$x\in\left(f^{p^e-1}\right)^{[1/p^e]}$. This completes the proof of
our claim. Note that Conjecture~\ref{conj} is satisfied in this
case.
\end{example}

\begin{example}\label{homogeneous}
We consider the case of a homogeneous polynomial
$f\in\ZZ[x_1,\ldots,x_n]$ of degree $d$ that defines a hypersurface
$Y$ in $\PP_{\ZZ}^{n-1}$. We assume that $Y$ is nonsingular over
$\CC$, that is, $f$ has an isolated singular point at the origin.
Let $p\gg 0$. The usual computation of multiplier ideals using the
blowing up at the origin shows that if $d\leq n$, then
$\cJ(f^{\lambda})_p=\FF_p[x_1,\ldots,x_n]$ for every $\lambda<1$,
and if $d>n$, then
\begin{equation}\label{formula1}
\cJ(f^{\lambda})_p = \left\{
\begin{array}{cl}
\FF_p[x_1,\ldots,x_n], & {\rm for}\,0 \le \lambda < \frac{n}{d}; \\[2mm]
(x_1,\ldots,x_n), & {\rm for}\, \frac{n}{d} \le \lambda <\frac{n+1}{d} ; \\[2mm]
\vdots & \vdots \\[2mm]
(x_1,\ldots,x_n)^{d-n}, & {\rm for}\,\frac{d-1}{d} \le t < 1.
\end{array}\right.
\end{equation}

Consider now an arbitrary prime $p$. We want to describe when
$\tau(f_p^{\lambda})$ is given by the same formula as the reduction
of the multiplier ideals. We again distinguish two cases, acoording
to he value of $d$.

\noindent{\bf Case 1}. Suppose that $d\leq n$. One can show that in
this case we have $\tau(f_p^{\lambda})=\FF_p[x_1,\ldots,x_n]$ for
every $\lambda<1$ if and only if the morphism induced by the
Frobenius
\begin{equation}\label{Frobenius_splitting}
F\colon H^{n-2}(Y_p,\omega_{Y_p})\to
H^{n-2}(Y_p,\omega_{Y_p}^{\otimes p})
\end{equation}
 is injective (recall that $\dim(Y_p)=n-2$, hence $H^{n-2}(Y_p,\omega_{Y_p})\simeq
 \FF_p)$.

\noindent{\bf Case 2}. Suppose now that $d>n$. Since we are only
interested in large values of $p$, we may assume that $p$ does
not divide $d$, and we fix $e$ such that $p^{e}\equiv 1$ (mod $d$).
For $0\leq r\leq d-n$, consider the morphism
$$T_r=f^{\frac{n+r}{d}(p^e-1)}F^e\colon H^{n-1}(\PP^{n-1}_{\FF_p},\cO(-n-r))\to H^{n-1}(\PP^{n-1}_{\FF_p},\cO(-n-r)),$$
where we denote by $F$ the morphism induced on the cohomology of the
projective space by the Frobenius ($T_r$ depends on the choice of
$e$, but if we replace $e$ by $me$, then $T_r$ is replaced by
$T_r^m$). Note that $H^{n-2}(Y_p,\cO(d-n-r)) \subseteq
H^{n-1}(\PP^{n-1}_{\FF_p},\cO(-n-r))$ consists of the elements
annihilated by $f$. In particular, when $r=d-n$, $T_r$ induces a map
from $H^{n-2}(Y_p,\cO_{Y_p})$ to itself, which coincides with the
one induced by the $e^{\rm th}$ iterate of the Frobenius on $Y_p$.

It is easy to see that we always have $\tau(f_p^{\lambda})\subseteq
(x_1,\ldots,x_n)^r$ if $\lambda\in [(n+r-1)/d, (n+r)/d)$. Moreover,
this is an equality for all such $\lambda$ if and only if the above
map $T_r$ is injective (in fact, it is enough to check the
injectivity of this map only on $H^{n-2}(Y_p,\cO(d-n-r))$). In
particular, Conjecture~\ref{conj} for $\lambda\in [(d-1)/d,1)$
predicts that there are infinitely many primes $p$ such that the map
induced by the Frobenius
$$H^{n-2}(Y_p,\cO_{Y_p})\to H^{n-2}(Y_p,\cO_{Y_p})$$
is injective.
\end{example}

\begin{remark}${\rm (}$\cite{MTW}, Example 4.6${\rm )}$
Conjecture~\ref{conj} holds in the case $d<n$ in the above example
by a standard argument. Indeed, in that case $Y$ is a Fano variety,
and it is known that if $p\gg 0$, then $Y_p$ is Frobenius split
(see, for example, Exercise 1.6.E(4) in \cite{BK}). Moreover, $Y_p$
is Frobenius split if and only if the morphism
(\ref{Frobenius_splitting}) is injective.

On the other hand, the case $d=n$ already seems very hard. One case
that is understood is when $d=n=3$ (that is, when $Y$ is an elliptic
curve). We see that in this case we have
$\tau(f_p^{\lambda})=\FF_p[x_1,x_2,x_3]$ for every $\lambda<1$ if
and only if $Y_p$ is ordinary. The behavior when $p$ varies depends
on whether $Y$ has complex multiplication or not. When $Y$ has
complex multiplication, then there is a quadratic field $K$ such
that if $Y_p$ is nonsingular, then $Y_p$ is ordinary if and only if
$p$ is completely split in $K$. On the other hand, if $Y$ does not
have complex multiplication, then by a result of Serre \cite{Serre}
the set of primes $p$ for which $Y_p$ is not ordinary has density
zero (note also that Elkies \cite{Elkies} proved that there are
infinitely many such primes). However, in this case too there is a
number field $K$ such that whenever $p$ is completely split in $K$,
the curve $Y_p$ is ordinary. This follows by taking first a finite
extension $K'$ of $\QQ$ containing all $\ell$-torsion points of $Y$,
where $\ell$ is an odd prime. Then one can show that if $p\neq 2,3,
\ell$ is a prime that is completely split in $K'$, then $Y_p$ is
ordinary (see Exercise 5.11 in \cite{Silverman}). It is enough to
take $K$ a finite extension of $K'$ in which $2$, $3$, and $\ell$
are not completely split.
\end{remark}

\begin{remark}
Motivated by Example~\ref{cusp} and the above remark (see also
\cite{MTW} for other examples) one can ask whether in the context of
Conjecture~\ref{conj} one can always find a number field $K$ such
that whenever $p$ is completely split in $K$, we have
$\tau(f_p^{\lambda})=\cJ(f^{\lambda})_p$. This would give a positive
answer to the conjecture by ${\rm \check{C}ebotarev}$'s density
theorem. The advantage of such a statement is that, in particular,
it would imply that the intersection of two such sets is again
infinite: if $K$ is a finite extension of two number fields $K_1$
and $K_2$, then whenever $p$ is completely split in $K$, it is
completely split also in $K_1$ and $K_2$.
\end{remark}

\section{The action of Euler operators in positive characteristic}

{}From now on we work in the following setup. Let $R$ be an $F$-finite regular
domain of positive characteristic $p$. We
denote by $D_R\subseteq {\rm End}_{\FF_p}(R)$ the ring of (absolute)
differential operators on $R$. In order to avoid the possible
confusion with the product in $D_R$, we denote the action of $P\in
D_R$ on $h\in R$ by $P\bullet h$. Since $R$ is an $F$-finite regular
ring, $D_R$ admits the following description (see \cite{Bli}, Proposition 3.3). For
every $e\geq 0$, let $D_R^{e}={\rm End}_{R^{p^e}}(R)$, in particular
$D_R^{0}=R$. We have $D_R^{e}\subseteq D_R^{e+1}$ and
$$D_R=\bigcup_{e\in\NN}D_R^{e}.$$

We also consider the polynomial ring $R[t]$, which is again a
regular $F$-finite domain. The corresponding rings of differential
operators will be denoted by $D_{R[t]}$ and $D_{R[t]}^e$. For every
$m\geq 1$, the divided power differential operator
$\partial_t^{[m]}$ acts on $R[t]$ by
$$\partial_t^{[m]}\bullet
at^r=a{{r}\choose {m}}t^{r-m}$$ for every $a\in R$ (we follow the
usual convention that ${r\choose m}=0$ if $r<m$). If $e$ is a
nonnegative integer, then
$$D_{R[t]}^e=D_R^e[t,\partial_t^{[m]}\mid m<p^e]$$
(it is enough to consider only those $\partial_t^{[m]}$ with $m$ a
power of $p$).  For $m\geq 1$, we put
$\theta_m:=\partial_t^{[m]}t^m$. In particular, $\theta_1=\partial_tt$ is the
Euler operator that appeared in \S 2.

We will repeatedly use the well-known theorem of Lucas
(see \cite{Lucas}, and also \cite{Granville}): if we 
consider the $p$-adic decompositions of $m$ and $n$, that is, 
$m=\sum_{i=0}^ra_ip^i$ and $n=\sum_{i=0}^rb_ip^i$, where $a_i,
b_i\in\{0,\ldots,p-1\}$, then
$${{m}\choose {n}}\equiv \prod_{i=0}^r{{a_i}\choose {b_i}}\,\,({\rm
mod}\,p).$$ For future reference, we collect in the following lemma
some computations in $D_{R[t]}$. They are standard and at least some of them are well-known,
but we include a proof for the benefit of the reader.

\begin{lemma}\label{lem5_2}
We have the following identities:
\begin{enumerate}
\item[i)] $[t,\theta_m]=-\theta_{m-1}\cdot t$ for every $m\geq 1$
$($with the convention that $\theta_0=1$$)$.
\item[ii)] $[\partial_t^{[p^e]},t^{p^e}]=1$ for every $e\geq 0$.
\item[iii)] $(\partial_t^{[p^e]})^r(t^{p^e})^r=\prod_{j=0}^{r-1}\left(\theta_{p^e}+j\right)$
\item[iv)] $\frac{(sr)!}{(s!)^r}\partial_t^{[sr]}=\left(\partial_t^{[s]}\right)^r$.
\item[v)] For every $i$ and $j$, we have ${{i+j}\choose
{i}}\partial_t^{[i+j]}=\partial_t^{[i]}\partial_t^{[j]}$.
\item[vi)] For every $i$ and $j$, we have $[\theta_i,\theta_j]=0$.
\item[vii)] If $a_0,\ldots,a_e\in\{0,\ldots,p-1\}$, and
$m=\sum_{i=0}^e a_ip^i$, then
$$\theta_m=\prod_{i=0}^e\frac{1}{a_i!}\cdot \prod_{j=0}^{a_i-1}(\theta_{p^i}+j),$$
where if $a_i=0$ the product $\prod_{j=0}^{a_i-1}(\theta_{p^i}+j)$ is understood to be $1$.
\item[viii)] For every $i$ and $j$ we have
\[
[\partial_t^{[p^i]}, \theta_p^j]=\left\{
\begin{array}{cl}
\partial_t^{[p^i]}, & \text{if}\,\,i=j; \\[2mm]
0, & \text{otherwise}.
\end{array}\right.
\]
\end{enumerate}
\end{lemma}

\begin{proof}
In order to prove i), it is enough to show that both sides give the
same result when applied to a monomial $t^n$, for $n\geq 0$. Note
that $\theta_m\bullet t^r={{m+r}\choose r} t^r$. Therefore we have
$$[t,\theta_m]\bullet t^n=\left({{m+n}\choose {m}}-{{m+n+1}\choose
{m}}\right)t^{n+1}=-{{m+n}\choose
{m-1}}t^{n+1}=-\theta_{m-1}t\bullet t^n.$$

We similarly deduce ii):
$$[\partial_t^{[p^e]},t^{p^e}]\bullet t^n=\left({{n+p^e}\choose
{p^e}}-{{n}\choose {p^e}}\right)t^n=t^n,$$ where the second equality
follows from Lucas' Theorem. Moreover, ii) implies by induction on
$r$ that
$$[\partial_t^{[p^e]}, (t^{p^e})^r]=r(t^{p^e})^{r-1}.$$
As a consequence, we easily get iii), also by induction on $r$.

The formulas in iv) and v) follow, too, by evaluating both sides on
every $t^n$:
$$\frac{(sr)!}{(s!)^r}\partial_t^{[sr]}\bullet t^n=\frac{(sr)!}{(s!)^r}
{n\choose {sr}}t^{n-rs}={n\choose s}\cdot {{n-s}\choose
s}\cdots{{n-(r-1)s}\choose{s}}t^{n-rs}=\left(\partial_t^{[s]}\right)^r\bullet
t^n,$$
$${{i+j}\choose
{i}}\partial_t^{[i+j]}\bullet
t^n={{i+j}\choose{i}}{{n}\choose{i+j}}t^{n-(i+j)}={{n-j}\choose
{i}}{{n}\choose{j}}t^{n-(i+j)}=\partial_t^{[i]}\partial_t^{[j]}\bullet
t^n.$$ To get vi), note that
$$\theta_i\theta_j\bullet t^n={{n+i}\choose {i}} {{n+j}\choose
{j}}t^{n}=\theta_j\theta_i\bullet t^n.$$

We now show vii). Note that by Lucas' Theorem, for every $0\leq
i\leq e$ and every $a\in\{1,\ldots,p-1\}$ we have in $\FF_p$
$${{a_ip^i+\cdots+a_ep^e}\choose {a_ip^i}}=1\,\,\,\,\,{\rm and}\,\,\,\,\frac{(ap^e)!}{(p^e!)^a}=\prod_{j=1}^a
{{jp^e}\choose {p^e}}=\prod_{j=1}^aj=a!.$$ Using this and v), iv),
iii), plus the fact that $[\partial_t^{[p^i]},t^{p^j}]=0$ whenever
$i<j$, we get
$$\partial_t^{[m]}t^m=\left(\prod_{i=0}^e\partial_t^{[a_{e-i}p^{e-i}]}\right)\cdot
t^m=\prod_{i=0}^e\frac{\left(\partial_t^{[p^{e-i}]}\right)^{a_{e-i}}}{a_{e-i}!}\cdot
\prod_{i=0}^et^{a_{e-i}p^{e-i}}$$
$$=\prod_{i=0}^e\frac{1}{a_{e-i}!}\left(\partial_t^{[p^{e-i}]}\right)^{a_{e-i}}(t^{p^{e-i}})^{a_{e-i}}
=\prod_{i=0}^e\frac{1}{a_i!}\cdot
\prod_{j=0}^{a_i-1}(\theta_{p^i}+j).$$
To avoid trivial special cases, the above products can be taken to run over those $i$ such that $a_{e-i}\neq 0$.

In order to prove viii), we evaluate both sides on $t^m$. 
Since $\partial_t^{[q]}\bullet t^m={{m}\choose
{q}}t^{m-q}$ and $\theta_q\bullet t^m={{q+m}\choose{q}}t^m$ for
every $m$ and $q$, we deduce that
$$[\partial_t^{[p^i]},\theta_{p^j}]\bullet t^m=
{{m}\choose{p^i}}\left({{m+p^j}\choose{p^j}}-{{m+p^j-p^i}\choose{p^j}}\right)t^{m-p^i}.$$

If we write $m+p^j=b_0+b_1p+\cdots$, with all
$b_i\in\{0,\ldots,p-1\}$, then ${{m+p^j}\choose {p^j}}=b_j$ in
$\FF_p$ (this is a consequence of Lucas' Theorem).
We deduce that
$${{m+p^j}\choose {p^j}}-{{m}\choose{p^j}}\equiv 1\,({\rm mod}\,p),$$
which gives our assertion when $i=j$.

Suppose now that $i\neq j$. We may assume that $p^i\leq m$, since otherwise 
${{m}\choose {p^i}}=0$. 
 If $i>j$, then the coefficients of $p^j$
in the $p$-adic expansions of $m+p^j$ and $m+p^j-p^i$ are the same,
hence ${{m+p^j}\choose {p_j}}={{m+p^j-p^i}\choose {p^j}}$ in
$\FF_p$. On the other hand, if $i<j$, then the coefficients of $p^i$
in the $p$-adic expansions of $m$ and $m+p^j$ are the same. Then
either they are equal to zero, in which case ${{m}\choose{p^i}}=0$,
or they are positive, and then $m+p^j$ and $m+p^j-p^i$ have the same
coefficient of $p^j$ in their $p$-adic expansion. In either case, we
get
$${{m}\choose{p^i}}\left({{m+p^j}\choose{p^j}}-{{m+p^j-p^i}\choose{p^j}}\right)\equiv 0
\,({\rm mod}\,p).$$
\end{proof}

It is easy to deduce from Lemma~\ref{lem5_2} the fact that the operators
$\theta_1,\theta_p,\ldots,\theta_{p^{e-1}}$ admit a common basis of
eigenvectors on every $D_{R[t]}^e$-module.

\begin{proposition}\label{decomp1}
If $M$ is a $D_{R[t]}^e$-module, then there is a unique
decomposition
$$M=\bigoplus_{i_1,\ldots,i_e\in\FF_p}M_{i_1,\ldots,i_e},$$
where for $1\leq \ell\leq e$, the operator $\theta_{p^{\ell-1}}$
acts on $M_{i_1,\ldots,i_e}$ by $-i_{\ell}$. Moreover, each
$M_{i_1,\ldots,i_e}$ is a $D_R^e$-module, and every morphism of
$D_{R[t]}^e$-modules preserves this decomposition.
\end{proposition}

\begin{proof}
Assertion iii) in the lemma implies that
$$\prod_{j=0}^{p-1}(\theta_{p^e}+j)=0$$
for every $e\geq 0$. Indeed, it is enough to show that
$(\partial_t^{[p^e]})^p=0$, and this follows from iv), since
$\frac{p^{e+1}!}{(p^e!)^p}$ is divisible by $p$.

Moreover, it follows from vi) that the $\theta_{p^e}$ are pairwise
commuting operators. This gives the existence of the decomposition
in the proposition, and the other assertions are immediate.
\end{proof}

\begin{remark}\label{rem_lower_e}
If $M$ is a $D_{R[t]}^e$-module, then $M$ is in
particular a $D_{R[t]}^{e-1}$-module, hence we get a
corresponding decomposition as such. It is clear that these decompositions
are compatible, that is
$$M_{i_1,\ldots,i_{e-1}}=\bigoplus_{j\in\FF_p}M_{i_1,\ldots,i_{e-1},j}.$$
\end{remark}

\begin{remark}\label{rem_other_theta}
If $M$ is as above, and if $m=b_1+b_2p+\cdots+b_ep^{e-1}$, where
all $b_i\in\{0,\ldots,p-1\}$, then $\theta_m$ acts on
$M_{i_1,\ldots,i_e}$ by
$$\prod_{\ell=1}^e(-1)^{b_{\ell}}{{i_{\ell}}\choose{b_{\ell}}}.$$
This is a consequence of the formula in Lemma~\ref{lem5_2} vii).
\end{remark}

\begin{proposition}\label{prop5_3}
If $M$ is a $D_{R[t]}^e$-module, then for every
$1\leq\ell\leq e$ we have
\begin{enumerate}
\item[i)] $\partial_t^{[p^{\ell-1}]}\cdot M_{i_1,\ldots,i_e}
\subseteq M_{i_1,\ldots,i_{\ell}+1,\ldots,i_e}$.
\item[ii)] \[
t^{p^{\ell-1}}\cdot M_{i_1,\ldots,i_e} \subseteq \left\{
\begin{array}{cl}
M_{i_1,\ldots,i_{\ell}-1\ldots,i_e}, & \text{if}\,\,i_{\ell}\neq 0; \\[2mm]
M_{i_1,\ldots,p-1,i_{\ell+1}-1,\ldots, i_e}, & \text{if}\,\,i_{\ell}=0,\,i_{\ell+1}\neq 0; \\[2mm]
\vdots & \vdots \\[2mm]
M_{i_1,\ldots,p-1,\ldots,p-1,i_e-1}, & \text{if}\,\,i_{\ell}=\ldots=i_{e-1}=0,\,i_e\neq 0; \\[2mm]
M_{i_1,\ldots,i_{\ell-1},p-1,\ldots,p-1}, &
\text{if}\,\,i_{\ell}=\ldots=i_e=0.
\end{array}\right.
\]
\end{enumerate}
\end{proposition}

\begin{proof}
The first formula follows from Lemma~\ref{lem5_2} viii). For the second
assertion, it is enough to consider the case $\ell=1$, since the
general case follows applying this one $p^{\ell-1}$ times. Note
first that by Remark~\ref{rem_other_theta}, for every $1\leq e'\leq e$ the
operator $\theta_{p^{e'}-1}$ is described on each component by
\[
\theta_{p^{e'}-1}\vert_{M_{j_1,\ldots,j_e}}=\left\{
\begin{array}{cl}
{\rm Id}\vert_{M_{j_1,\ldots,j_e}}, & \text{if}\,\,j_1=\ldots=j_{e'}=p-1; \\[2mm]
0, & \text{otherwise}.
\end{array}\right.
\]
Let $w\in M_{i_1,\ldots,i_e}$. We show by induction on
$\ell\in\{1,\ldots,e\}$ that
\[
\theta_{p^{\ell-1}}(tw)=\left\{
\begin{array}{cl}
-(i_{\ell}-1)tw, & \text{if}\,\,i_1=\ldots=i_{\ell-1}=0; \\[2mm]
-i_{\ell}tw, & \text{otherwise}
\end{array}\right.
\]
(with the convention that when $\ell=1$ we are always in the first
case). This implies the assertion in ii) for multiplication by $t$.

By Lemma~\ref{lem5_2} i), we have
\begin{equation}\label{eq1}
\theta_{p^{\ell-1}}(tw)=t\theta_{p^{\ell-1}}(w)+\theta_{p^{\ell-1}-1}(tw)
\end{equation}
(with the convention that $\theta_0=1$). This gives
$\theta_1(tw)=-(i_1-1)tw$. Suppose now that we know the formula for
$\theta_{p^{\ell'-1}}(tw)$ for all $\ell'\leq\ell-1$. In particular,
this implies that $tw\in M_{j_1,\ldots,j_{\ell-1}}$ for some $j_1,\ldots,j_{\ell-1}$, 
and $j_1=\ldots=j_{\ell-1}=p-1$ if and only if $i_1=\ldots=i_{\ell-1}=0$.
Our description of $\theta_{p^{\ell-1}-1}$ gives
\[ \theta_{p^{\ell-1}-1}(tw)=\left\{
\begin{array}{cl}
tw, & \text{if}\,\,i_1=\ldots=i_{\ell-1}=0; \\[2mm]
0, & \text{otherwise}.
\end{array}\right.
\]
The formula for $\theta_{p^{\ell-1}}(tw)$ now follows from this and
(\ref{eq1}). The proof of ii) is now complete.
\end{proof}

\begin{remark}\label{rem5_11}
It follows from Proposition~\ref{prop5_3} ii) that for every
$D_{R[t]}^e$-module $M$ and every
$i_1,\ldots,i_e\in\FF_p$ the component $M_{i_1,\ldots,i_e}$ is a
$D_R^e[t^{p^e}]$-submodule.
\end{remark}

\begin{example}\label{structure_sheaf}
If we write $m=\sum_{i\geq 1}a_ip^{i-1}$, with $0\leq a_i\leq p-1$,
then we have seen that $\theta_{p^e}\bullet t^m={{m+p^e}\choose
p^e}t^m=(a_e+1)t^m$. It follows that if $M=R[t]$, then
for every $a_1,\ldots,a_e\in\{0,\ldots,p-1\}$, the component
$R[t]_{a_1,\ldots,a_e}$ of $R[t]$ is
free over $R[t^{p^e}]$ with basis $t^m$, where
$m=\sum_{i=1}^e(p-1-a_i)p^{i-1}$.
\end{example}

\section{The $D$-module $B_f$ in positive characteristic}

We now specialize the discussion in the previous section to the
case of the module $B_f$. 
Suppose that $f\in R$ is nonzero. By analogy with the situation
in \S 2, we put
$$B_f:=R[t]_{f-t}/R[t].$$
Since $R[t]$ is naturally a $D_{R[t]}$-module, and since every
localization of a $D_{R[t]}$-module is again a $D_{R[t]}$-module, we
see that $B_f$ has a natural structure of $D_{R[t]}$-module. We want to study
the decomposition of $B_f$ under the action of Euler operators.

In order to describe this decomposition we will make use of the fact that $B_f$ is a unit
$F$-module. We start with a lemma that applies to arbitrary unit $F$-modules.  
 For the theory of unit $F$-modules
we refer to \cite{Lyu} or \cite{Bli}.
Let $R[t]^{(e)}$ denote the $R[t]$-bimodule $R[t]$, with the left
module structure being the usual one, and the right one being induced by the
$e^{\rm th}$ iterated Frobenius. 
A \emph{unit $F$-module} over $R[t]$ is an $R[t]$-module $M$, together with
a map $F\colon M\to M$ that
is semilinear with respect to the Frobenius morphism on $R[t]$, and
such that the induced $R[t]$-linear map $\nu_1\colon
R[t]^{(1)}\otimes_{R[t]}M\to M$ given by $\nu_1(h\otimes
w)=hF(w)$ is an isomorphism. Iterating, we get isomorphisms
$$\nu_e\colon R[t]^{(e)}\otimes_{R[t]} M\to M$$
for every $e\geq 1$. Note that $R[t]^{(e)}\otimes_{R[t]}M$ has a
natural $D_{R[t]}^e$-module structure such that $P\cdot (h\otimes
w)=(P\bullet h)\otimes w$. 
It follows that a unit $F$-module $M$ over $R[t]$ has a canonical
$D_{R[t]}^e$-module structure
 such that $\nu_e$ is an isomorphism of
$D_{R[t]}^e$-modules. In fact, letting $e$ vary one gets a
$D_{R[t]}$-module structure on $M$.

\begin{lemma}\label{lem_F_module1}
For every unit $F$-module $M$ over $R[t]$, and every
$i_1,\ldots,i_e\in\{0,\ldots,p-1\}$, the component
$M_{i_1,\ldots,i_e}$ is generated as an $R$-module by $t^m
F^e(M)$, where $m=\sum_{\ell=1}^e(p-i_{\ell}-1)p^{\ell-1}$.
\end{lemma}

\begin{proof}
Since $\nu_e$ is an isomorphism of $D_{R[t]}^e$-modules, it induces an
isomorphism between the corresponding components of the two
$D_{R[t]}^e$-modules. Therefore every element in $M_{i_1,\ldots,i_e}$ 
can be written as $\nu_e(h\otimes w)=hF^e(w)$, for some 
$h\in R[t]_{i_1,\ldots,i_e}$. We now deduce our
assertion from Example~\ref{structure_sheaf}.
\end{proof}

\begin{corollary}\label{cor_F_module2}
If $M$ is a unit $F$-module over $R[t]$, then $F(M_{i_1,\ldots,i_e})
\subseteq M_{p-1,i_1,\ldots,i_e}$.
\end{corollary}

Let $M$ be a unit $F$-module over $R[t]$. Given $w\in M$,  for every $e\geq 1$ and every
$i_1,\ldots,i_e\in\{0,\ldots,p-1\}$, we put
$$w_{i_1,\ldots,i_e}:=\nu_e(t^{\sum_{\ell=1}^e(p-i_{\ell}-1)p^{\ell-1}}\otimes w)=t^{\sum_{\ell=1}^e(p-i_{\ell}-1)p^{\ell-1}}F^e(w)\in M.$$
It follows from Lemma~\ref{lem_F_module1} that
$w_{i_1,\ldots,i_e}\in M_{i_1,\ldots,i_e}$. Note that the
induced map $M\to M_{i_1,\ldots,i_e}$ that takes each $w$ to
$w_{i_1,\ldots,i_e}$ is semilinear with respect to the $e^{\rm th}$
iterate of the Frobenius morphism on $R[t]$.

\bigskip

We now turn to the case of the module $B_f$. 
The $D_{R[t]}$-module structure on $B_f$ is induced by a unit $F$-module structure,
such that the $R[t]$-linear isomorphism $\nu_1\colon
R[t]^{(1)}\otimes_{R[t]}B_f\to B_f$ is given by $\nu_1(a\otimes
u)=au^{p}$. Therefore the induced isomorphism $\nu_e$ satisfies $\nu_e(a\otimes u)=
au^{p^e}$.

Note that $B_f$ is a free $R$-module with basis $\{\delta_m\}_{m\geq
0}$, where $\delta_m$ is the class of $\frac{1}{(f-t)^{m+1}}$ in
$B_f$. A special role is played by $\delta:=\delta_0$.
It follows by direct computation that for every $e\geq 0$ we
have
\begin{equation}\label{eq4_1}
t^{p^e}\cdot
\delta_m=f^{p^e}\delta_m-\delta_{m-p^e}\,\,\,(\delta_i=0\,\,\text{for}\,\,i<0)
\end{equation}
\begin{equation}\label{eq4_2}
\partial_t^{[p^e]}\cdot \delta_m={{m+p^e}\choose{p^e}}\delta_{m+p^e}.
\end{equation}

Suppose now that $e\geq 1$ is fixed, and consider $i_1,\ldots,i_e\in\{0,\ldots,p-1\}$,
and a nonnegative integer $m$. We put
$$Q_{i_1,\ldots,i_e}^m:=(\delta_m)_{i_1,\ldots,i_e}=
\nu_e\left(t^{\sum_{\ell=1}^e(p-i_{\ell}-1)p^{\ell-1}}\otimes\delta_{m}\right).$$
We will see that when $m$ varies, these elements give an $R$-basis of $(B_f)_{i_1,\ldots,i_e}$.
We start by writing these elements in the basis given by the $\delta_i$.

\begin{lemma}\label{lem4_2}
With the above notation, for every
$i_1,\ldots,i_e\in\{0,1,\ldots,p-1\}$ and every nonnegative integer
$m$ we have
\begin{equation}\label{eq4_4}
Q_{i_1,\ldots,i_e}^m=(-1)^{i_1+\cdots+i_e}\sum_{j_1,\ldots,j_e}{{i_1+j_1}\choose
{i_1}}\cdots{{i_e+j_e}\choose{i_e}}f^{\sum_{\ell=1}^ej_{\ell}p^{\ell-1}}\delta_{mp^e+(i_1+j_1)+\cdots+(i_e+j_e)p^{e-1}},
\end{equation}
where the sum is over the integers $j_1,\ldots,j_e$ such that $0\leq
j_{\ell}\leq p-i_{\ell}-1$ for all $\ell$.
\end{lemma}

\begin{proof}
The right-hand side of (\ref{eq4_4}) is equal to
$\nu_e(h\otimes\delta_{m})$, where
$$h=(-1)^{i_1+\cdots+i_e}\prod_{\ell=1}^e\left(\sum_{j_{\ell}=0}^{p-i_{\ell}-1} {{i_{\ell}+j_{\ell}}\choose{i_{\ell}}}
f^{j_{\ell}p^{\ell-1}}(f-t)^{p^{\ell-1}(p-i_{\ell}-j_{\ell}-1)}\right).$$
Consider now
$$h_{\ell}:=\sum_{j_{\ell}=0}^{p-i_{\ell}-1}{{i_{\ell}+j_{\ell}}\choose{i_{\ell}}}f^{j_{\ell}}(f-t)^{p-i_{\ell}-j_{\ell}-1}.$$
It follows from Lemma~\ref{lem4_3} below that we may write in the
fraction field of $R[t]$
$$h_{\ell}=(f-t)^{p-1-i_{\ell}}\sum_{j_{\ell}=0}^{p-i_{\ell}-1}
{{i_{\ell}+j_{\ell}}\choose{i_{\ell}}}\left(\frac{f}{f-t}\right)^{j_{\ell}}=
(f-t)^{p-1-i_{\ell}}\cdot\left(1-\frac{f}{f-t}\right)^{p-1-i_{\ell}}=(-t)^{p-1-i_{\ell}}.$$
We deduce that
$h=\prod_{\ell=1}^e\left(t^{p-1-i_{\ell}}\right)^{p^{\ell-1}}$,
which implies the formula in the lemma.
\end{proof}

\begin{lemma}\label{lem4_3}
We have the following identity in the polynomial ring $\FF_p[x]$
$$\sum_{j=0}^{p-i-1}{{i+j}\choose i}x^j=(1-x)^{p-i-1}$$
for every $i\in\{0,\ldots,p-1\}$.
\end{lemma}

\begin{proof}
We have
$$\sum_{j=0}^{p-i-1}{{i+j}\choose
i}x^j=\frac{1}{i!}\left(\sum_{j=0}^{p-1}x^j\right)^{(i)}
=\frac{1}{i!}\left(\frac{1-x^p}{1-x}\right)^{(i)}=
\frac{1}{i!}\left((1-x)^{p-1}\right)^{(i)}$$
$$=(-1)^i\frac{(p-1)(p-2)\cdots (p-i)}{i!}(1-x)^{p-1-i}=(1-x)^{p-1-i}.$$
\end{proof}

We can now describe the decomposition of $B_f$ under the action of the Euler operators.

\begin{theorem}\label{thm4_1}
For every $e\geq 1$, and $i_1,\ldots,i_e\in\{0,\ldots,p-1\}$, the set
$\{Q_{i_1,\ldots,i_e}^m\vert m\geq 0\}$ gives an $R$-basis of $(B_f)_{i_1,\ldots,i_e}$.
Moreover, if $1\leq\ell\leq e$, then the following hold:
\begin{enumerate}
\item[i)] $\partial_t^{[p^{\ell-1}]}\cdot
Q_{i_1,\ldots,i_e}^m=-(i_{\ell}+1)Q_{i_1,\ldots
i_{\ell}+1,\ldots,i_e}^m$ $($when $i_{\ell}=p-1$, this expression is
understood to be zero$)$.
\item[ii)]
\[
t^{p^{\ell-1}}\cdot Q_{i_1,\ldots,i_e}^m = \left\{
\begin{array}{cl}
Q_{i_1,\ldots,i_{\ell}-1\ldots,i_e}^m, & \text{if}\,\,i_{\ell}\neq 0; \\[2mm]
Q_{i_1,\ldots,p-1,i_{\ell+1}-1,\ldots, i_e}^m, & \text{if}\,\,i_{\ell}=0,\,i_{\ell+1}\neq 0; \\[2mm]
\vdots & \vdots \\[2mm]
Q_{i_1,\ldots,p-1,\ldots,p-1,i_e-1}^m, & \text{if}\,\,i_{\ell}=\ldots=i_{e-1}=0,\,i_e\neq 0; \\[2mm]
f^{p^e}Q_{i_1,\ldots,i_{\ell-1},p-1,\ldots,p-1}^m-
Q_{i_1,\ldots,i_{\ell-1},p-1,\ldots,p-1}^{m-1}, &
\text{if}\,\,i_{\ell}=\ldots=i_e=0
\end{array}\right.
\]
$($where we put $Q^{-1}_{j_1,\ldots,j_e}=0$ for every
$j_1,\ldots,j_e$$)$.
\item[iii)] $R\cdot Q_{i_1,\ldots,i_e}^m$ is a $D_{R}^e$-submodule
of $B_f$, isomorphic to $R$ by an isomorphism that takes
$Q_{i_1,\ldots,i_e}^m$ to $1$.
\end{enumerate}
\end{theorem}

\begin{proof}
We claim that the $Q_{i_1,\ldots,i_e}^m$, when $i_1,\ldots,i_e$, and $m$ vary,
give an $R$-basis of $B_f$.  Indeed, we see that in Lemma~\ref{lem4_2},
the term in (\ref{eq4_4}) corresponding to
$j_1=\ldots=j_e=0$ is
$$(-1)^{i_1+\ldots+i_e}\delta_{mp^e+i_1+\cdots+i_ep^{e-1}},$$ and all
the other terms are in the linear span of the
$\delta_{mp^e+i'_1+\cdots+i'_ep^{e-1}}$, where $i_{\ell}\leq
i'_{\ell}\leq p-1$ for all $\ell$, and $i'_{\ell}>i_{\ell}$ for some
$\ell$. Since  the $\delta_i$ with $i\geq 0$ give  an $R$-basis of $B_f$, we
deduce our claim. Since each $Q_{i_1,\ldots,i_e}^m$ lies in $(B_f)_{i_1,\ldots,i_e}$,
we get the first assertion in the theorem.

If $P\in D_{R[t]}^e$, we may compute $P\cdot Q_{i_1,\ldots,i_e}^m$
as $\nu_e\left(P\bullet
t^{\sum_{\ell=1}^e(p-i_{\ell}-1)p^{\ell-1}}\otimes\delta_{m}\right)$.
If $P\in D_R^e$ and $h\in \FF_p[t]\subseteq R[t]$, then $P\bullet
h=h(P\bullet 1)$. Therefore $P\cdot Q_{i_1,\ldots,i_e}^m=(P\bullet
1)Q_{i_1,\ldots,i_e}^m$, which implies iii).

Note that if $1\leq \ell\leq e$, and if we write
$n=i_1+\cdots+i_ep^{e-1}+mp^e$, with
$i_1,\ldots,i_e\in\{0,\ldots,p-1\}$ and $m\geq 0$, then
$$\partial_t^{[p^{\ell-1}]}\bullet
t^n={n\choose{p^{\ell-1}}}t^{n-p^{\ell-1}}=i_{\ell}t^{n-p^{\ell-1}}$$
(the second equality follows from Lucas' Theorem). If we take $n=
\sum_{\ell'=1}^e(p-i_{\ell'}-1)p^{\ell'-1}$, then we get
$$\partial_t^{[p^{\ell-1}]}\cdot
Q_{i_1,\ldots,i_e}^m=\nu_e\left(\left(\partial_t^{[p^{\ell-1}]}\bullet
t^n\right)\otimes\delta_m\right)=
(p-i_{\ell}-1)\nu_e(t^{n-p^{\ell-1}}\otimes\delta_m)=-(i_{\ell}+1)Q_{i_1,\ldots,i_{\ell}+1,\ldots,i_e},$$
hence i). We also have
$$t^{p^{\ell-1}}\cdot
Q_{i_1,\ldots,i_e}^m=\nu_e\left(t^{(p-i_1-1)+\ldots+(p-(i_{\ell}-1)+1)p^{\ell-1}+\ldots+(p-i_e-1)p^{e-1}}\otimes
\delta_{m}\right).$$ The formula in ii) is an immediate consequence.
\end{proof}

\begin{remark}
It follows from the formula in Lemma~\ref{lem4_2} that the
$Q_{i_1,\ldots,i_e}^m$ with $0\leq i_{\ell}\leq p-1$ for all $\ell$,
and with $m\leq m_0$, give an $R$-basis of the
$D_{R[t]}^e$-submodule $\bigoplus_{i\leq
(m_0+1)p^e-1}R\cdot\delta_i$.
\end{remark}

\begin{remark}\label{transformation}
It would be useful to have an explicit formula for the change of
basis when we replace $e$ by $e+1$. In the case $m=0$ we have the
following formula:
\begin{equation}\label{eq_transformation}
Q_{i_1,\ldots,i_e}^0=\sum_{j=0}^{p-1}(-1)^j{{p-1}\choose
{j}}f^{jp^e}Q_{i_1,\ldots,i_e,j}^0
\end{equation}
for every $i_1,\ldots,i_e\in\{0,1,\ldots,p-1\}$. Indeed, we have
$$Q_{i_1,\ldots,i_e}^0=t^{\sum_{\ell=1}^e(p-i_{\ell}-1)p^{\ell-1}}\cdot\frac{1}{(f-t)^{p^e}}
=t^{\sum_{\ell=1}^e(p-i_{\ell}-1)p^{\ell-1}}(f-t)^{p^e(p-1)}\cdot\frac{1}{(f-t)^{p^{e+1}}}$$
$$=\sum_{j=0}^{p-1}(-1)^{p-1-j}{{p-1}\choose{j}}f^{jp^e}t^{p^e(p-j-1)+\sum_{\ell=1}^e(p-i_{\ell}-1)p^{\ell-1}}\cdot
\frac{1}{(f-t)^{p^{e+1}}}$$ $$=\sum_{j=0}^{p-1}(-1)^j{{p-1}\choose
{j}}f^{jp^e}Q_{i_1,\ldots,i_e,j}^0.$$
\end{remark}

\section{Bernstein-Sato polynomials in positive characteristic}

We keep the notation in the previous section. Motivated by the
analogy with the situation described in \S 2, we study some modules
over rings of differential operators of $R$. For every positive
integer $e$, consider the $D_R^e$-submodule of $B_f$
$$M_{f}^e:=D_R^e[\theta_1,\theta_p,\ldots,\theta_{p^{e-1}}]\cdot \delta.$$
The union of all $M_f^e$ is the $D_R$-module
$$M_f:=\underrightarrow{\rm lim}_e M_f^e=D_R[\theta_{p^i}\mid i\geq
0]\cdot\delta.$$
 We use the decomposition in
Theorem~\ref{thm4_1} to give an explicit description of  $M_f^e$.

\begin{proposition}\label{prop5_1}
With the above notation, we have
\begin{equation}\label{eq5_2}
M_{f}^e=\bigoplus_{i_1,\ldots,i_e=0}^{p-1}\left(D_R^e\bullet
f^{i_1+i_2p+\cdots+i_{e}p^{e-1}}\right) Q_{i_1,\ldots,i_e}^0.
\end{equation}
\end{proposition}

\begin{proof}
We first show that
\begin{equation}\label{eq5_3}
\delta=\sum_{i_1,\ldots,i_e=0}^{p-1}(-1)^{i_1+\cdots+i_e}
{{p^e-1}\choose{i_1+i_2p+\cdots+i_{e}p^{e-1}}}f^{i_1+i_2p+\cdots+i_ep^{e-1}}Q_{i_1,\ldots,i_e}^0.
\end{equation}
To see this, note that
$$\delta=\nu_e\left((f-t)^{p^e-1}\otimes\delta\right)$$
$$=\nu_e\left( \sum_{i_1,\ldots,i_e=0}^{p-1}{{p^e-1}\choose
i_1+i_2p+\cdots+i_ep^{e-1}} f^{i_1+i_2p+\cdots+i_ep^{e-1}}\cdot
(-t)^{(p-1-i_1)+\cdots+p^{e-1}(p-1-i_e)}\otimes\delta\right)
$$
$$=\sum_{i_1,\ldots,i_e=0}^{p-1}(-1)^{i_1+\cdots+i_e}{{p^e-1}\choose
{i_1+i_2p+\cdots+i_ep^{e-1}}}f^{i_1+\cdots+i_e p^{e-1}}
Q_{i_1,\ldots,i_e}^0.$$

Note now that the binomial coefficients in (\ref{eq5_3}) are all
different from zero. Indeed, it follows from Lucas' Theorem that
$${{p^e-1}\choose{i_1+i_2p+\cdots+i_{e}p^{e-1}}}\equiv \prod_{\ell=1}^e
{{p-1}\choose i_{\ell}}\,\,({\rm mod}\,p).$$
 By
Theorem~\ref{thm4_1}, each $R\cdot Q_{i_1,\ldots,i_e}^0$ is an
eigenspace of $\theta_{p^{\ell-1}}$ with eigenvalue $-i_{\ell}$, and
therefore $M_{f}^e$ is the direct sum of its intersections with the
$R\cdot Q_{i_1,\ldots,i_e}^0$. Since we have an isomorphism of
$D_R^e$-modules $R\simeq R\cdot Q_{i_1,\ldots,i_e}^0$ that takes $1$
to $Q_{i_1,\ldots,i_e}^0$, we get the decomposition (\ref{eq5_2}).
\end{proof}

\begin{remark}\label{rem5_1}
It is easy to show that the $D_R^e$-submodules of $R$ are precisely
the ideals of the form $J^{[p^e]}$, for some ideal $J$ of $R$ (see,
for example, Lemma 2.2 in \cite{BMSm1}). Using the notation in \S 3,
we see that for every $g\in R$
$$D_R^e\bullet g=\left(g^{[1/p^e]}\right)^{[p^e]}.$$
\end{remark}

\begin{remark}\label{cor5_1}
The subring $D_R^e[\theta_1,\theta_p,\ldots,\theta_{p^{e-1}}]$ of
$D_{R[t]}^e$ contains all $\theta_m$ with $m<p^e$. This is an
immediate consequence of the formula in Lemma~\ref{lem5_2} vii).
\end{remark}

We will be interested in the action of the operators
$\theta_1,\theta_p\ldots,\theta_{p^{e-1}}$ on the quotient
$M_{f}^e/t M_{f}^e$. The following lemma shows that indeed,
$tM_f^e\subseteq M_f^e$, and the above operators have an induced
action on the quotient module.

\begin{lemma}\label{lem5_1}
For every $e\geq 1$ we have
$$t M_{f}^e=D_R^e[\theta_1,\theta_p,\ldots,\theta_{p^{e-1}}]\cdot f\delta\subseteq M_{f}^e.$$
\end{lemma}

\begin{proof}
It is clear that we have
$D_R^e[\theta_1,\ldots,\theta_{p^{e-1}}]\cdot f\delta\subseteq
M_{f}^e$. Note also that $t\delta=f\delta$, hence it is enough
to prove the equality in $D_{R[t]}^e$
$$t\cdot D_R^e[\theta_1,\theta_p,\ldots,\theta_{p^{e-1}}]=
D_R^e[\theta_1,\theta_p\ldots,\theta_{p^{e-1}}]\cdot t.$$
Lemma~\ref{lem5_2} i) and Remark~\ref{cor5_1} give
$$[t,\theta_{p^i}]=-\theta_{p^i-1}\cdot t\in
D_R^e[\theta_1,\theta_p,\ldots,\theta^{p^{i-1}}]t$$ for every $i\leq
e$. Since $t$ commutes with the operators in $D_R^e$, we deduce by
induction on $i\leq e-1$ that
$$t\cdot D_R^e[\theta_1,\theta_p,\ldots,\theta_{p^i}]\subseteq
D_R^e[\theta_1,\theta_p,\ldots,\theta_{p^i}]\cdot t.$$ 
The reverse inclusion follows similarly, using the fact that for every $m$ we have 
$[t,\theta_m]=-t\cdot\sum_{j=0}^{m-1}\theta_j$
(recall that $\theta_0=1$). This assertion follows in turn from 
Lemma~\ref{lem5_2} i), by induction on $m$.
\end{proof}

\begin{corollary}\label{cor5_2}
For every positive integer $e$ we have a decomposition
$$M_f^e/t M_f^e=\bigoplus_{i_1,\ldots,i_e}W_{i_1,\ldots,i_e},$$ such that
for every $1\leq\ell\leq e$, the operator $\theta_{p^{\ell-1}}$ acts
on $W_{i_1,\ldots,i_e}$ by $-i_{\ell}$, and
$$W_{i_1,\ldots,i_e}\simeq (D_R^e\bullet
f^{i_1+i_2p+\cdots+i_ep^{e-1}})/(D_R^e\bullet
f^{1+i_1+i_2p+\cdots+i_ep^{e-1}})$$ $($the $i_1,\ldots,i_e$ vary
over $\{0,\ldots,p-1\}$$)$.
\end{corollary}

\begin{proof}
The assertion follows from Proposition~\ref{prop5_1} and
Theorem~\ref{thm4_1} ii) and iii).
\end{proof}

\begin{notation}
Let $\Gamma^e_f\subseteq\{0,1,\ldots,p-1\}^e$ be the set of those
$(i_1,\ldots,i_e)$ such that $W_{i_1,\ldots,i_e}\neq\emptyset$. In
other words, $(i_1,\ldots,i_e)\in\Gamma_f^e$ if and only if there is
a nonzero element $u\in M_f^e/t M_f^e$ such that
$(\theta_{p^{\ell-1}}+i_{\ell})u=0$ for $1\leq\ell\leq e$.
\end{notation}

By analogy with the characteristic zero case, we define the
Bernstein-Sato polynomial of $f$ to be the minimal polynomial of
$-\theta_1$ on the $D_R^1$-module $M_f^1/t M_f^1$. In other words,
we have
$$b_f(s)=\prod_{i\in\Gamma^1_f}(s-i)\in \FF_p[s].$$
Note that unlike in characteristic zero, this polynomial always has
distinct roots.

In order to also keep track of the action of the higher
$\theta_{p^e}$, we introduce the \emph{higher Bernstein-Sato
polynomials} $b_f^{(e)}(s)\in\QQ[s]$, defined by
$$b_f^{(e)}(s)=\prod_{(i_1,\ldots,i_e)\in\Gamma^e_f}\left(s-\left(\frac{i_e}{p}+\cdots+
\frac{i_1}{p^e}\right)\right).$$ Note that $b_f\in \FF_p[s]$ while
$b_f^{(1)}\in\QQ[s]$, but they contain the same amount of information. It
follows from definition that $b_f^{(e)}$ has distinct roots, all of
them in $\frac{1}{p^e}\ZZ\cap [0,1)$. Our next goal is to relate the
roots of $b_f^{(e)}$ to the $F$-jumping exponents of $f$.

\begin{theorem}\label{thm_main2}
For every $e\geq 1$, the roots of the Bernstein-Sato polynomial
$b_f^{(e)}(s)$ are simple, and they are given by the rational
numbers $\frac{\lceil p^e\lambda\rceil-1}{p^e}$, where $\lambda$
varies over the $F$-jumping exponents of $f$ in $(0,1]$.
\end{theorem}

Before giving the proof of the theorem, we introduce some notation.
Given $\lambda\in (0,1]$, we can write it uniquely as
\begin{equation}
\lambda=\sum_{i\geq 1}\frac{c_i(\lambda)}{p^i},
\end{equation}
with all $c_i(\lambda)\in\{0,1,\ldots,p-1\}$, and such that
infinitely many of the $c_i(\lambda)$ are nonzero. Note that the
$c_i(\lambda)$ are determined recursively by $c_1(\lambda)=\lceil
\lambda p\rceil-1$ and $c_i(\lambda)=
c_{i-1}(p\lambda-c_1(\lambda))$ for $i\geq 2$. Moreover, for every
$e$ we have
\begin{equation}\label{eq_thm_main2}
\frac{c_1(\lambda)}{p}+\cdots+\frac{c_e(\lambda)}{p^e}=\frac{\lceil
\lambda p^e\rceil -1}{p^e}.
\end{equation}

\begin{lemma}\label{lem_thm_main2}
For every positive integer $e$,
$$\Gamma_f^e=\{(c_e(\lambda),\ldots,c_1(\lambda))\mid
\lambda\in(0,1]\,\text{is an F-jumping exponent for}\,f\}.$$
\end{lemma}

\begin{proof}
It follows from Corollary~\ref{cor5_2} that
$(i_1,\ldots,i_e)\in\{0,1,\ldots,p-1\}^{e}$ lies in $\Gamma_f^e$ if
and only if $D_R^e\bullet f^{i_1+i_2p+\cdots+i_ep^{e-1}}\neq
D_R^e\bullet f^{1+i_1+i_2p+\cdots+i_ep^{e-1}}$. On the other hand,
for every nonnegative integer $m$ we have
$$D_R^e\bullet
f^m=\left((f^m)^{[1/p^e]}\right)^{[p^e]}=\tau(f^{m/p^e})^{[p^e]}$$
(for the second equality see, for example, Lemma~2.1 in
\cite{BMSm1}).

 Since the Frobenius morphism on $R$ is flat, for
every two ideals $I_1$ and $I_2$ in $R$ we have
$I_1^{[p^e]}\subseteq I_2^{[p^e]}$ if and only if $I_1\subseteq
I_2$. Therefore $(i_1,\ldots,i_e)\in\Gamma_f^e$ if and only if there
is an $F$-jumping exponent $\lambda$ of $f$ in the interval
$\left(\frac{m}{p^e},\frac{m+1}{p^e}\right]$, where
$m=i_1+i_2p+\cdots+i_ep^{e-1}$. On the other hand, it follows from
the definition of the $c_j(\lambda)$ that this is the case if and
only if
$m=c_1(\lambda)p^{e-1}+\cdots+c_{e-1}(\lambda)p+c_e(\lambda)$. Using
the fact that $i_{\ell}, c_j(\lambda)\in\{0,1,\ldots,p-1\}$, it
follows that
 this is further equivalent with
 $(i_1,\ldots,i_e)=(c_e(\lambda),\ldots,c_1(\lambda))$,
 which completes the proof of the lemma.
 \end{proof}

\begin{proof}[Proof of Theorem\ref{thm_main2}]
The fact that the roots of $b_f^{(e)}$ are simple is a consequence
of the definition. Lemma~\ref{lem_thm_main2} implies that these
roots correspond to the rational numbers of the form
$\frac{c_1(\lambda)}{p}+\cdots+\frac{c_e(\lambda)}{p^e}$, where
$\lambda$ varies over the $F$-jumping exponents of $f$.
Formula~(\ref{eq_thm_main2}) implies the statement of the theorem.
\end{proof}

\bigskip

\begin{remark}\label{finiteness}
It follows from Theorem~1.1 in \cite{BMSm1} that there are finitely
many (say $r$) $F$-jumping exponents of $f$ in $(0,1]$.
Theorem~\ref{thm_main2} implies that the number of roots of
$b_f^{(e)}$ is bounded above by $r$ for every $e$, with equality for
$e\gg 0$.
\end{remark}

\begin{remark}\label{F_thresholds}
We can use the interpretation of the $F$-jumping exponents as
$F$-thresholds (see Remark~\ref{interpretation_thresholds}) to
reinterpret Theorem~\ref{thm_main2} as follows. Let $J$ be a proper
ideal of $R$ containing $f$ (this is equivalent with $c^J(f)\leq
1$). For a given $e\geq 1$, the ratio $\frac{\nu^{J}(p^e)}{p^e}$ is
a root of $b_f^{(e)}$, and all roots of $b_f^{(e)}$ are of this form
(for some ideal $J$).
\end{remark}

\begin{example}\label{case_1}
If $f$ is not invertible, then $1$ is an $F$-jumping exponent for
$f$: use Remark~\ref{interpretation_thresholds} and the fact that
$c^{(f)}(f)=1$. Since $c_i(1)=p-1$ for every $i$, we see that
$(p-1,\ldots,p-1)\in\Gamma_e$ for every $e\geq 1$. Therefore
$1-\frac{1}{p^e}$ is always a root of $b_f^{(e)}$.
\end{example}

\begin{remark}\label{rem5_5}
It follows from Lemma~\ref{lem_thm_main2} that we have a surjective
map $\Gamma_f^{e+1}\to\Gamma_f^e$ that takes $(i_1,\ldots,i_{e+1})$
to $(i_2,\ldots,i_{e+1})$. Note that we have another map
$\Gamma_f^{e+1}\to\Gamma_f^e$, taking $(i_1,\ldots,i_{e+1})$ to
$(i_1,\ldots,i_e)$. Indeed, by the same lemma, it is enough to show
that for every $F$-jumping coefficient $\lambda\in (0,1]$ for $f$,
we have $(c_{e+1}(\lambda),\ldots,c_2(\lambda))\in\Gamma^e_f$.

It it is known that if $\lambda$ is an $F$-jumping exponent of $f$,
then the fractional part $\{p\lambda\}$ of $p\lambda$ is also an
$F$-jumping exponent (see Proposition~3.4 in \cite{BMSm2}). If
$p\lambda$ is not an integer, then
$c_i(\{p\lambda\})=c_{i+1}(\lambda)$ for $i\geq 1$, hence
$(c_{e+1}(\lambda),\ldots,c_2(\lambda))\in\Gamma^e_f$. On the other
hand, if $p\lambda=m\in\ZZ$, then $c_1(\lambda)=m-1$, and
$c_i(\lambda)=p-1$ for $i\geq 2$. In this case, we get
$(c_{e+1}(\lambda),\ldots,c_2(\lambda))\in\Gamma_f^e$ by
Example~\ref{case_1}.
\end{remark}

\begin{remark}\label{comparison}
Note that we have canonical maps $\phi_e\colon M_f^e/t M_f^e\to
M_f^{e+1}/t M_f^{e+1}$. If we denote by $Q'_{i_1,\ldots,i_e}$ the
class of $f^{i_1+i_2p+\cdots i_ep^{e-1}}Q_{i_1,\ldots,i_e}^0$ in
$M_f^e/t M_f^e$, then it follows from Remark~\ref{transformation}
that
$$\phi_e(Q'_{i_1,\ldots,i_e})=\sum_{j=0}^{p-1}(-1)^j{{p-1}\choose {j}}
Q'_{i_1,\ldots,i_e,j}.$$ We will see in Example~\ref{cusp3} below
that it can happen that no map $\phi_e$ is injective, and that
we miss a lot of information if instead of considering all $M_f^e$
we consider only $M_f$.
\end{remark}

\begin{example}\label{cusp0}
Consider the case of the cusp $f_p=x^2+y^3\in \FF_p[x,y]$, with
$p>3$. We have seen in Example~\ref{cusp} that the only jumping
numbers of $f_p$ in $(0,1]$ are $c_p$ and $1$, where
$c_p=\frac{5}{6}$ if $p\equiv 1$ (mod $3$), and
$c_p=\frac{5}{6}-\frac{1}{6p}$ if $p\equiv 2$ (mod $3$). Note that we have
$1=\sum_{e\geq 1}(p-1)\cdot\frac{1}{p^e}$ and
\[
c_p= \left\{
\begin{array}{cl}
\sum_{i\geq 1}\frac{5(p-1)}{6}\cdot\frac{1}{p^i}, & \text{if}\,\,p\equiv 1\,({\rm mod}\,3); \\[2mm]
\frac{5p-7}{6}\cdot\frac{1}{p}+\sum_{i\geq
2}(p-1)\cdot\frac{1}{p^i}, & \text{if}\,\,p\equiv 2\,({\rm mod}\,3).
\end{array}\right.
\]
It follows from Lemma~\ref{lem_thm_main2} that if $p\equiv 1$ (mod
$3$), then
$$\Gamma_{f_p}^e=\left\{(p-1,\ldots,p-1),\left(\frac{5(p-1)}{6},\ldots,\frac{5(p-1)}{6}\right)\right\}$$
for every $e\geq 1$, and if $p\equiv 2$ (mod $3$), then
$$\Gamma_{f_p}^1=\left\{p-1,\frac{5p-7}{6}\right\},
\,\,\Gamma_{f_p}^e=\left\{(p-1,\ldots,p-1),\left(p-1,\ldots,p-1,\frac{5p-7}{6}\right)\right\}
\,{\rm for}\,e\geq 2.$$ We deduce the formula for the Bernstein-Sato
polynomial
\[ b_f^{(e)}(s)= \left\{
\begin{array}{cl}
\left(s-\left(1-\frac{1}{p^e}\right)\right)\left(s-\frac{5}{6}\left(1-\frac{1}{p^e}\right)\right), & \text{if}\,\,p\equiv 1\,({\rm mod}\,3); \\[2mm]
\left(s-\left(1-\frac{1}{p^e}\right)\right)\left(s-\left(\frac{5p-1}{6p}-\frac{1}{p^e}\right)\right),
& \text{if}\,\,p\equiv 2\,({\rm mod}\,3).
\end{array}\right.
\]
In particular, we see that in $\FF_p[s]$ we have
\[
b_f(s)= \left\{
\begin{array}{cl}
\left(s-(p-1)\right)\left(s-\frac{5}{6}(p-1)\right)=(s+1)\left(s+\frac{5}{6}\right), & \text{if}\,\,p\equiv 1\,({\rm mod}\,3); \\[2mm]
\left(s-(p-1)\right)\left(s-\frac{5p-7}{6}\right)=(s+1)\left(s+\frac{7}{6}\right),
& \text{if}\,\,p\equiv 2\,({\rm mod}\,3).
\end{array}\right.
\]
\end{example}

\begin{example}\label{cusp3}
Let $f_p=x^2+y^3\in\FF_p[x,y]$, where $p>3$ is a prime with $p\equiv
2$ (mod $3$). Using the notation in Remark~\ref{comparison}, the
computation in the previous example shows that for every $e\geq 2$
we have
$$M_{f_p}^e/t M_{f_p}^e=D_R^e \cdot Q'_{p-1,\ldots,p-1,p-1}\oplus D_R^e\cdot
Q'_{p-1,\ldots,p-1,\frac{5p-7}{6}},$$ and both components are
nonzero. We have
$$h_e(Q'_{p-1,\ldots,p-1})=Q'_{p-1,\ldots,p-1,p-1}+(-1)^{(5p-7)/6}{{(p-1)}\choose{(5p-7)/6}}Q'_{p-1,\ldots,p-1,\frac{5p-7}{6}}$$
and $h_e(Q'_{p-1,\ldots,p-1,\frac{5p-7}{6}})=0$. In particular, the images
of all $Q'_{p-1,\ldots,p-1}$ in $M_{f_p}/t M_{f_p}$ coincide,
and this element generates $M_{f_p}/t M_{f_p}$ over $D_R$. 
We deduce that all operators $\theta_{p^e}$ (for $e\geq 0$) are equal
to the identity on $M_{f_p}/t M_{f_p}$.
\end{example}

\begin{example}\label{quadric}
Let $f=x_1^2+\cdots+x_n^2\in \FF_p[x_1,\ldots,x_n]$, where $p>2$ and
$n\geq 2$. It follows from Example 4.1 in \cite{MTW} that the only
$F$-jumping exponent of $f$ in $(0,1]$ is $1$. Therefore
$b_f^{(e)}=\left(s-\left(1-\frac{1}{p^e}\right)\right)$ for every
$e\geq 1$. In particular, we have $b_f(s)=(s+1)$. Note however that if
$\widetilde{f}=x_1^2+\cdots+x_n^2\in\QQ[x_1,\ldots,x_n]$, then 
$b_{\widetilde{f}}(s)=(s+1)\left(s+\frac{n}{2}\right)$, see \cite{Kashiwara3}, Example~6.19.
\end{example}

\begin{example}\label{quasihomogeneous}
Let $f\in R=k[x_1,\ldots,x_n]$, with $k$ an $F$-finite field of characteristic $p>0$, and suppose that there are integers
$d$ and $w_1,\ldots,w_n$ such that for every monomial
$x^u=x_1^{u_1}\cdots x_n^{u_n}$ with nonzero coefficient in $f$, we
have $\sum_iu_iw_i\equiv d$ (mod $p$). We assume that $d\not\equiv
0$ (mod $p$), hence we can write $f=\frac{1}{d}\cdot \sum_{i=1}^n
w_ix_i\frac{\partial f}{\partial x_i}$. Therefore $f$ has isolated
singularities if and only if
 $\dim_k(R/J_f)<\infty$, where $J_f=(\partial f/\partial
x_1,\ldots,\partial f/\partial x_n)$. If this is the case, then for
every root $\beta\neq -1$ of $b_f$ there is a monomial $x^u\not\in
J_f$ such that $\beta=-\frac{\sum_iw_i(u_i+1)}{d}$.

The argument is similar to the corresponding one in characteristic
zero (see \S 6.4 in \cite{Kashiwara3}). It is clear that we have an isomorphism
$$M_f/t M_f\simeq D_R^1[\partial_tt]/J,$$
where $J=\{P\in D_R^1[\partial_tt] \mid P\cdot \delta\in t M_f\}$. If we put
$T_f:=(1-\partial_tt)\cdot M_f/tM_f$, then $b_f(s)/(s+1)$ is the
minimal polynomial of $-\partial_tt$ on $T_f$. Moreover, we have
$T_f\simeq D_R^1[\partial_tt]/J'$, where $J'=\{Q\in D_R[\partial_tt] \mid
(1-\partial_tt)Q\in J\}$.

 Let $\xi=\sum_iw_i
x_i\partial_i$, where we put $\partial_i:=
\partial_{x_i}$.
It follows by direct computation that
$(\xi+d\partial_tt)\cdot\delta=0$, hence
$\partial_tt+\frac{1}{d}\xi\in J$. Moreover, since $f\in J$ and
$$\left(\partial_i f+\frac{\partial f}{\partial
x_i}(\partial_tt-1)\right)\cdot\delta=0,$$ we conclude that
$\partial f/\partial x_i\in J'$ for every $i$. Hence we have a
surjection of $k$-vector spaces
$$D_R^1[\partial_tt]/D_R^1(\partial_tt+\frac{1}{d}\xi,\partial f/\partial x_1,\ldots,\partial f/\partial
x_n)\simeq
k[\partial_1,\ldots,\partial_n]/(\partial_1^p,\ldots,\partial_n^p)\otimes_k
R/J_f\to T_f.$$ In order to describe the action of $\-\partial_tt$
on the left-hand side, note first that this commutes with the
operators $\partial_i$. Furthermore, we have in this quotient module
$$(-\partial_tt)\cdot x^u=\frac{1}{d}x^u\xi=\frac{1}{d}\left(\sum_iw_i
\partial_ix_i\right)x^u-\frac{1}{d}\sum_i(u_i+1)w_ix^u,$$ and
therefore
$$(-\partial_tt)\cdot (1\otimes x^u)+\frac{\sum_i(u_i+1)w_i}{d}(1\otimes x^u)\in \sum_j\partial_j\cdot
k[\partial_1,\ldots,\partial_n]/(\partial_1^p,\ldots,\partial_n^p)\otimes_k
R/J_f.$$ 
It follows that if we consider on $k[\partial_1,\ldots,\partial_n]/(\partial_1^p,\ldots,\partial_n^p)\otimes_k
R/J_f$ the decreasing filtration by the vector subspaces $\{W^{\ell}\otimes_kR/J_f\}_{\ell}$, where $W^{\ell}=(\partial_1^{\ell},\ldots,\partial_n^{\ell})/
(\partial_1^p,\ldots,\partial_n^p)$, then for every $g\otimes x^u\in W^{\ell}\otimes_k R/J_f$
we have 
$$(-\partial_tt)\cdot (g\otimes x^u)+\frac{\sum_i(u_i+1)w_i}{d}(g\otimes x^u)
\in W^{\ell+1}\otimes_k R/J_f.$$
This implies that every eigenvalue of $-\partial_tt$ on $T_f$ is of the form $-\frac{\sum_i(u_i+1)w_i}{d}$,
for some monomial $x^u\in R\smallsetminus J_f$.
\end{example}

\bigskip

We end by raising some questions related to the setup considered in this paper.

\begin{question}
The discreteness of the set of $F$-jumping exponents of $f$ is
equivalent with the fact that there is some $r$ such that
$\#\Gamma_f^e\leq r$ for every $e$. The rationality of these
exponents is a direct consequence of their discreteness (see Theorem
3.1 in \cite{BMSm2}). On the other hand, discreteness plus
rationality implies the eventual periodicity of the components of the elements of the sets
$\Gamma_f^e$, when $e$ varies. Is is possible to make a stronger periodicity
statement for the modules $M_f^e$ ?
\end{question}

\begin{question}
In characteristic zero, the main application of the Bernstein-Sato
polynomial in the setting that we discussed is the construction of
the $V$-filtration. Is there an analogue of the $V$-filtration in
positive characteristic ? A related question is the following:
suppose that $\widetilde{f}\in\ZZ[x_1,\ldots,x_n]$. Is it possible to lift the $V$-filtration 
of $\widetilde{f}$ to a filtration on $\ZZ[x_1,\ldots,x_n,t]_{f-y}/\ZZ[x_1,\ldots,x_n,t]$ ? If this is the case,
what can be said about the reduction modulo $p$ of this filtration, for $p\gg 0$ ?
Note that a minimum requirement for the $V$-filtration over $\ZZ$ would be ``to 
put the operator $\partial_t t$ in upper-triangular form". More optimistically, one can ask
about the existence of a structure that would deal at the same time with all operators
$\partial_t^{[m]}t^m$, with $m\geq 1$.
\end{question}

\begin{question}
As in characteristic zero, one can consider the Bernstein-Sato
polynomial of $f$ with respect to an arbitrary element $w\in B_f$. These invariants seem
to be particularly relevant when $w=h\delta$, for some $h\in R$. In
this case they contain the same amount of information as the sets
$$\Gamma_{f,w}^e:=\{(i_1,\ldots,i_e)\in\{0,\ldots,p-1\}^e\mid
D_R^e\bullet hf^{i_1+i_2p+\cdots+i_ep^{e-1}}\neq D_R^e\bullet
hf^{1+i_1+i_2p+\cdots+i_ep^{e-1}}\}.$$ For example, a natural
question is whether the numbers $\#\Gamma_{f,w}^e$ are all bounded
above by some $r$. Moreover, are these numbers eventually constant ?

In characteristic zero, the construction of the $V$-filtration 
is based on the existence of $b_f$ and on the
rationality of its roots. On the other hand, once the existence of the $V$-filtration is known, then 
the existence of all $b_{f,w}$, and the rationality of their roots follow.
Is it possible, in positive characteristic, to use the eventual periodicity
of the components of the elements 
of the sets $\Gamma_f^e$, to prove a similar result about the sets
$\Gamma_{f,w}^e$ ?
\end{question}

\providecommand{\bysame}{\leavevmode \hbox \o3em
{\hrulefill}\thinspace}


\end{document}